\documentclass[15pt,reqno]{amsart}
\usepackage{amsmath,amssymb,amsthm}
\usepackage[latin1]{inputenc}
\usepackage{version,tabularx,multicol}
\usepackage{graphicx,float, url}
\usepackage{amsfonts}

\usepackage{graphicx}
\usepackage{epsfig}
\usepackage{subfigure}
\usepackage{color}
\usepackage{float}
\usepackage{color}
\usepackage{flafter}

\vspace{0.3cm}

\theoremstyle{plain}
\newtheorem{theo}{Theorem}[section]

\newtheorem{cor}[theo]{Corollary}
\newtheorem{prop}[theo]{Proposition}

\newtheorem{lem}[theo]{Lemma}

\theoremstyle{remark}
\newtheorem{rem}[]{Remark}[section]

\newcommand{\cl}{{\mathcal L}}

\newcommand{\cp}{{\mathcal P}}

\newcommand{\E}{{\mathbb E}}

\newcommand{\N}{{\mathbb N}}
\renewcommand{\P}{{\mathbb P}}

\newcommand{\ind}{{\bf 1}}

\newcommand{\Var}{{\rm Var}}

\begin{document}

\title{Asymptotics
of the minimal clade size and related functionals of certain beta-coalescents}

\date{\today}

\author{Arno Siri-J\'egousse}
\address{Universidad de Guanajuato, Departamento de Matem\'aticas.
Calle Jalisco s/n
Col. Mineral de Valenciana
36240 Guanajuato, Guanajuato
Mexico.}
\email{arno@cimat.mx}

\author{Linglong Yuan}
\address{Universit\'e Paris 13, Sorbonne Paris Cit\'e, LAGA, CNRS, UMR 7539, F-93430, Villetaneuse, France.}
\email{yuan@math.univ-paris13.fr}

\begin{abstract}
{This article shows the asymptotics of distributions of various functionals
 of the Beta$(2-\alpha,\alpha)$ $n$-coalescent process with $1<\alpha<2$ when
 $n$ goes to infinity.
{This process} is a Markov process taking {values} in the set of partitions of
$\{1, \dots, n\}$, evolving from the intial value $\{1\},\cdots, \{n\}$  by {merging (coalescing) blocks together into one} and finally  {reaching} the absorbing state $\{1, \dots, n\}$. 
The minimal clade
of $1$ is the block which contains $1$ at the time of
{coalescence of the singleton $\{1\}$.}
The limit size of the minimal clade of $1$ is provided.
To this, we express it as a function of the {coalescence} time of $\{1\}$ and sizes of blocks at that time.
 Another quantity concerning the size of the largest block (at {deterministic small} time and at the {coalescence} time of $\{1\}$) is also studied. }
\end{abstract}

\thanks{Linglong Yuan  benefited from the support of the ``Agence Nationale de la Recherche'': ANR MANEGE (ANR-09-BLAN-0215). }


\maketitle
\section{Introduction and main results}
Coalescent theory was initiated by Kingman (\cite{Kin82a, Kin82b, Kin82c}) to model the genealogical tree of a sample of $n$ individuals of a certain population. The {so-called} Kingman $n$-coalescent is a {continuous-time Markov chain taking values in $\cp_n$}, the set of partitions of $\N_n=\{1,2,\cdots,n\}$. 
It starts from $n$ singletons $\{1\}, \{2\}, \cdots, \{n\}$ representing $n$ individuals {(or lineages)} and at any time, each couple of blocks merges (or coalesces)  independently into one block at rate $1$. 
The process reaches almost surely in finite time the absorbing state $\{1,2,\cdots,n\}$  which is called MRCA (most recent common ancestor). 
Kingman showed that  the genealogy of a {sample of size $n$ in a population evolving according to the Cannings population model of size $N$} converges in the sense of finite dimensional distribution to the Kingman $n$-coalescent when $N$ goes to $\infty$, under some assumptions over the reproduction law in the Cannings model. 
Roughly speaking, it is required that one individual in the population should not have a lot of progenies so that its children occupy a large ratio of the next generation. 
This assumption happens to fail in some marine species (see \cite{He94}, \cite{El06},  \cite{Bo94}, \cite{Ar04}). To model this phenomenon, Pitman (\cite{Pit99}) and Sagitov (\cite{Sag99}) introduced at the same time the $\Lambda$ $n$-coalescent,
denoted by $\Pi^{(n)}=(\Pi^{(n)}(t),t\geq 0)$. It is characterized by a finite measure $\Lambda$ on $[0,1]$. 

{The process $\Pi^{(n)}$ is still a  continuous-time Markov chain starting from {$\{1\}, \{2\}, \cdots, \{n\}$}, but with the following dynamics:}
at any time $t\geq 0$, if $\Pi^{(n)}(t)$ has $b$ blocks ($b\geq 2$), then each $k$-tuple ($2\leq k\leq b$) of blocks coalesces together into one at rate 
\begin{equation}\label{lambda}\lambda_{b,k}:=\int_0^1x^{k-2}(1-x)^{b-k}\Lambda(dx).\end{equation}
As a consequence, the rate to the next coalescence is 
\begin{equation}\label{gb}g_b:=\sum_{k= 2}^b {\binom{b}{k}}\lambda_{b,k}.\end{equation}
In particular, the Kingman $n$-coalescent is a special $\Lambda$ $n$-coalescent with $\Lambda$ being the Dirac measure  at $0$.

Definition (\ref{lambda}) {is a reformulation of two properties of the process $\Pi^{(n)}$} {(see \cite{Pit99})}.
The first one is exchangeability. Let $\rho_{n}$ be a permutation of $\mathbb{N}_n$. The map $\rho_n$ induces naturally a map $\bar{\rho}_n$ on $\cp_n$.  Then we have
$$\bar{\rho}_n\circ \Pi^{(n)}\stackrel{(d)}{=}\Pi^{(n)}.$$
The second one is consistency. 
For any $2\leq m\leq n$, let $\varrho_{n,m}$ be the natural restriction from $\mathbb{N}_n$ to $\mathbb{N}_m$ and $\bar{\varrho}_{n,m}$ the induced map {from $\cp_n$ to $\mathcal{P}_m$}. Then 
$$\bar{\varrho}_{n,m}  \circ \Pi^{(n)}\stackrel{(d)}{=}\Pi^{(m)}.$$
These two properties also imply that one can build a projective limit process, the 
so-called $\Lambda$-coalescent, denoted by $\Pi=(\Pi(t),t\geq 0)$,
taking values in the set $\cp$ of partitions of $\mathbb{N}$.
 For any restriction $\varrho_{n}$ from $\mathbb{N}$ to $\mathbb{N}_n$ and its induced map $\bar{\varrho}_n$ from $\cp$ to $\cp_n$, 
$$\bar{\varrho}_{n}\circ \Pi\stackrel{(d)}{=}\Pi^{(n)}.$$
 
 The Beta$(2-\alpha,\alpha)$-coalescent 
 with $0< \alpha< 2$ is a special and important example of $\Lambda$-coalescents.
  In this case, $\Lambda$ is the Beta measure with parameters $2-\alpha$ and $\alpha$.  
  If $\alpha$ tends to $2$, then the limit process obtained is the Kingman coalescent. 
If $\alpha=1$, the process obtained is the celebrated Bolthausen-Sznitman coalescent (\cite{BS98}). 
This article deals {with the case $1<\alpha<2$ and,
for the sake of simplicity,} we will refer to Beta$(2-\alpha,\alpha)$-coalescent as Beta-coalescent.
This class of coalescent processes was introduced by Schweinsberg (\cite{Sch03}) and deeply studied in \cite{BBS07},\cite{BBS08}. 
In particular, in \cite{BBS08}, many results on the small-time behavior of various functionals of the Beta-coalescent are discovered.
{Meanwhile}, many asymptotic studies, motivated by biology, have been developed for the Beta $n$-coalescent
(see for example \cite{DDSJ08}, \cite{Ker12}, \cite{DFSY}, \cite{KIA12}) , when $n$ goes to $\infty$.

In this paper, we aim to study more asymptotic results on some functionals of
 the Beta $n$-coalescent with $1<\alpha<2$ when $n$ grows to $\infty$. {We denote 
$$A\sim B,$$
if $\frac{A}{B}$ tends {deterministically} or randomly to $1$ in the limit, depending on different contexts. Here $A,B$ can be functions, sequences of real values,  random variables.} Denote by $\stackrel{a.s.}{\to}$ the almost sure convergence and by $\stackrel{P}{\to}$ the convergence in probability.

{The length of the external branch of individual $i$, also called \textit{unicity of individual $i$} {by biologists} (\cite{RBY04}),
is denoted by $T_i^{(n)}$.
It is the coalescence time of $\{i\}$,} defined as follows
$$T_i^{(n)}:=\sup\{t: \{i\}\in \Pi^{(n)}(t)\}.$$
{The length of a randomly chosen external branch provides a measure} of the genetic variation of the population
since it gives some information on the ``distance" of an individual to the rest of the sample.
{Exchangeability of the coalescent implies that}
$$T_i^{(n)}\stackrel{(d)}{=}T_1^{(n)},\,\,1\leq i\leq n.$$
The law of $T_1^{(n)}$ has interested many people since the first article \cite{BF05} dealing with the Kingman coalescent case.
We give a short survey of results already discovered.
\begin{itemize}
\item Kingman:  $nT_1^{(n)}$ converges in distribution to a random variable with density $\frac{8}{(2+t)^3}\ind_{t\geq 0}$ (\cite{BF05},\cite{CNK07}).
\item Beta-coalescent with $1<\alpha<2$: $n^{\alpha-1}T_1^{(n)}\stackrel{(d)}{\longrightarrow}T$, where $T$ is a random variable with density  function $f_T$:
\begin{equation}\label{ft}f_T(t)=\frac{1}{(\alpha-1)\Gamma(\alpha)}(1+\frac{t}{\alpha\Gamma(\alpha)})^{-\frac{\alpha}{\alpha-1}-1}\ind_{t\geq 0}.\end{equation}
 (see \cite{DFSY} {where the result is stated in a more general case}). 
 \item $\displaystyle\lim_{n\to\infty}\frac{g_n}{n\mu^{(n)}}=0$ where $g_n$ is defined in (\ref{gb}) and $\mu^{(n)}=\int_{1/n}^1x^{-1}\Lambda(dx):$ $\mu^{(n)}T_1^{(n)}$ is asymptotically distributed as an exponential random variable with mean $1$ (\cite{Yuan13}).  This class of processes contains Beta-coalescents with $0<\alpha<1$ (see also \cite{Pit99}, \cite{M10}, \cite{GIM08} for other proofs) and the Bolthausen-Sznitman coalescent (see also \cite{DM13}, \cite{FM09} for other proofs). 
\end{itemize}

In this paper, we prove {the following}
\begin{theo}\label{th:manyt}
Consider a Beta $n$-coalescent with $1<\alpha<2$.  For any fixed $k\in\mathbb{N}$, as $n\to\infty$,\begin{equation}
n^{\alpha-1}(T_1^{(n)},\cdots, T_k^{(n)})\overset{(d)}{\longrightarrow} (T_1, \cdots, T_k),
\end{equation}
where $(T_i,i\in\N)$ are i.i.d. copies of $T$ with density (\ref{ft}).
\end{theo}
A similar result has been proved for Bolthausen-Sznitman coalescent (\cite{DM13}), but {the asymptotic independence} is not true for coalescents satisfying $\int_0^1x^{-1}\Lambda(dx)<\infty$ (\cite{M10}).

Let {$K^{(n)}=(K^{(n)}(t),t\geq 0)$ denote the block-counting process of $\Pi^{(n)}$, i.e., $K^{(n)}(t)$ stands for the number of blocks of the partition $\Pi^{(n)}(t)$ for $t\geq 0$.} Define 
$$Q^{(n)}:={K^{(n)}}((T_1^{(n)})_-)-{K^{(n)}}(T_1^{(n)})+1,$$
where $(T_1^{(n)})_-$ is the time just prior to $T_1^{(n)}$. 
{In other words, $Q^{(n)}$ is the number of blocks involved in the coalescence event of $\{1\}$ in $\Pi^{(n)}$}. 
\begin{theo}\label{alpha-1}
Consider a Beta $n$-coalescent with $1<\alpha<2$. $Q^{(n)}$ converges in law to a random variable $Q$ taking values in $\{2,3,\cdots\}$ such that for any $k\geq 2$\begin{equation}\label{pbn}
q_k:=\mathbb{P}(Q=k)=\frac{(\alpha-1)\Gamma(k-\alpha)}{\Gamma(k)\Gamma(2-\alpha)}.
\end{equation}  
{Furthermore}, $Q^{(n)}$ and $T_1^{(n)}$ are asymptotically independent. 
\end{theo}
Notice that in the Kingman coalescent,
$Q^{(n)}=Q=2$ almost surely.
{The following proposition shows that, for $\Lambda$-coalescents
satisfying $\int_{[0,1]}x^{-1}\Lambda(dx)<\infty$, $Q^{(n)}$ converges in probability to infinity.}

\begin{prop}\label{dust}
Consider a $\Lambda$ $n$-coalescent with the characteristic measure  satisfying $\int_0^1x^{-1}\Lambda(dx)<\infty$, then {$Q^{(n)}$ converges in probability to infinity}. 
\end{prop}
The above proposition is even true {in} the Bolthausen-Sznitman coalescent case (see Remark \ref{bsinf}). 

A quantity of interest in biology is the \textit{minimal clade size}.   
{It is the size of the minimal clade of a randomly chosen individual
(or of the individual 1, considered in this paper).
The minimal clade is the
block that contains $1$ at the time $\{1\}$ is coalesced.
The size of the minimal clade tells how many individuals share the genealogy with individual 1 after time $T_1^{(n)}$.
Let us denote the {minimal clade size} by $Y^{(n)}$.
In the Kingman case, Blum and Fran\c{c}ois (\cite{BF05}) showed that 
$$\mathbb{P}(Y^{(n)}=k)=\frac{4}{(k+1)k(k-1)}, k=2, \cdots, n-1; \,\, \mathbb{P}(Y^{(n)}=n)=\frac{2}{n(n-1)}.$$
Freund and Siri-J\'egousse \cite{FA13} studied the case of the Bolthausen-Sznitman coalescent. 
In this case
$$\frac{\ln Y^{(n)}}{\ln n}\stackrel{(d)}{\longrightarrow } U_{[0,1]},$$
where $U_{[0,1]}$ is a uniform variable over $[0,1]$.
Asymptotics of moments were also found.

We state out our result by at first giving some notations.
\begin{itemize}
\item Let $\mu$ be Slack's probability distribution on {$[0,\infty)$} (see \cite{Sla68}) characterized by its Laplace transform
\begin{equation}\label{eq:mulap}
\cl_\mu(\lambda)=\int_0^{\infty}e^{-\lambda x}\mu(dx)=1-(1+\lambda^{1-\alpha})^{-\frac{1}{\alpha-1}}, \quad\lambda \geq 0.
\end{equation}
\item {Define a 1-parameter family of laws (not a process)} $(\beta(t), t\geq 0)$ such that for any $k\geq 1$
\begin{align}\label{eq:beta}
\mathbb{P}(\beta(t)=k)&=\frac{1}{\Gamma(k)}\left(\frac{t}{\alpha\Gamma(\alpha)}\right)^\frac{k-1}{\alpha-1}\int_0^{\infty}e^{-x\left(\frac{t}{\alpha\Gamma(\alpha)}\right)^\frac{1}{\alpha-1}}x^{k}\mu(dx).
\end{align}
\end{itemize}

The following result could be regarded as a consequence of Theorem \ref{alpha-1}.
\begin{theo}\label{size1}
Consider a Beta $n$-coalescent with $1<\alpha<2$.  Let $(\beta_i(t), t\geq 0)_{i\geq 1}$ be i.i.d. copies of $(\beta(t),t\geq 0)$ {and $Q,T$ be random variables defined respectively in (\ref{pbn}) and (\ref{ft})}. Assume that $ (\beta_i(t), t\geq 0)_{i\geq 1}, Q, T$ are all independent. Then 
\begin{equation}
Y^{(n)}\stackrel{(d)}{\longrightarrow }Y=1+\sum_{i=1}^{Q-1}\beta_{i}(T).
\end{equation}
The law of $Y$ can be described {as follows}: for any $l\geq2,$
\begin{align*}&\displaystyle \mathbb{P}(Y=l)=\int_0^{\infty}\sum_{k=2}^{l}q_k\sum_{i_1+\cdots+i_{k-1}=l-1}\Big(\Pi_{j=1}^{k-1}\mathbb{P}(\beta(t)=i_j)\Big)f_T(t)dt.
\end{align*}
\end{theo}

Next, we establish a close relation between the random variable $Q$ and the family $(\beta(t),t\geq 0)$.
Notice that $\displaystyle \lim_{t\to0+}\mathbb{P}(\beta(t)=1)=1$.
\begin{prop}\label{2qk}$1)$ For any $k\geq 2$, 
\begin{equation}\label{qk=}q_k=(\alpha-1)\Gamma(\alpha)\lim_{t\to0+}\frac{\mathbb{P}(\beta(t)=k)}{t}.\end{equation}
$2)$ The Laplace transform of $Q$ is 
\begin{equation}\label{lq}\mathbb{E}[e^{-\lambda Q}]=\lim_{t\to 0+}\mathbb{E}[(\alpha-1)\Gamma(\alpha)\frac{e^{-\lambda \beta(t)}\ind_{\beta(t)\geq 2}}{t}]=e^{-\lambda}\Big(1-(1-e^{-\lambda})^{\alpha-1}\Big)\end{equation}
for any $\lambda\geq 0$.
\end{prop}

The law of $Y$ looks quite complicated, which may harm the applicability of the result. 
However the clarification given below could at some point improve the situation. 

\begin{cor}\label{yk}If $k$ tends to $\infty$, one has
$\mathbb{P}(Y>k)\sim \frac{\int_{0}^{\infty}t^{{\alpha-1}}f_T(t)dt}{((\alpha-1)\Gamma(\alpha))^{\alpha-1}\Gamma(1-(\alpha-1)^2)}k^{-(\alpha-1)^2}.$
\end{cor}

If $\alpha$ goes to $1$, $k^{-(\alpha-1)^2}$ goes to $1$. This is consistent with the Bolthausen-Sznitman case where $Y=\infty$ almost surely. If $\alpha$ tends to $2$, $k^{-(\alpha-1)^2}$ goes to $k^{-1}$. This is in fact not consistent with the law of $Y$ in the Kingman case. The corollary reveals some kind of ``discontinuity" 
{between the Beta-coalescent and the Kingman coalescent. }

The size of the block containing {one} specific integer evolves in an increasing way at different speed. 
It is clear that at time $(T_1^{(n)})_-$ the block containing 1 is still of size 1 while other blocks could have grown quite a lot.
One way to measure this speed is to consider the size of the largest block at time $T_1^{(n)}$. 
We denote this variable by $\tilde{W}^{(n)}$.  
The bigger $\tilde{W}^{(n)}$ is, the more inhomogeneous the speed is. 
To study $\tilde{W}^{(n)}$, we first consider the size of the largest block at any time $t$, denoted by
 $W^{(n)}(t)$. In this way, we have

{$$\tilde W^{(n)}=W^{(n)}(T_1^{(n)}).$$}
\begin{theo}\label{max}
Consider a Beta $n$-coalescent with $1<\alpha<2$, 
\begin{equation}
\frac{W^{(n)}((\alpha-1)\alpha\Gamma(\alpha)n^{1-\alpha}t)}{n^{\frac{1}{\alpha}}}\stackrel{(d)}{\longrightarrow }W(t),
\end{equation}where $W(t)$ is a positive random variable with a type-2 Gumbel law, i.e., for any $x\geq 0,$ $\displaystyle \mathbb{P}(W(t)\leq x)=e^{-x^{-{\alpha}}\frac{(\alpha-1)t}{\Gamma(2-\alpha)}}$.

\end{theo}

The methodology employed to prove the above theorem is similar to {that used in the proof of}  Proposition 1.6 in \cite{BBS08}, although there are some small differences.

The following result about $\tilde{W}^{(n)}$ happens to be a straightforward consequence of the above theorem. 

\begin{cor}\label{max1}
As $n$ tends to $\infty$, 
\begin{equation}
\frac{\tilde W^{(n)}}{n^{\frac{1}{\alpha}}}\stackrel{(d)}{\longrightarrow }\tilde W,
\end{equation}
where $\tilde{W}$ is a positive random variable such that for any $x\geq0,$ 
$$\mathbb{P}(\tilde W\leq x)=\int_0^{\infty}e^{-x^{-\alpha}{\frac{t}{\alpha\Gamma(\alpha)\Gamma(2-\alpha)}}}f_T(t)dt.$$
\end{cor}

This paper is organized as follows. In Section 2, we study external branch lengths and the block-counting process in small time and prove Theorem \ref{th:manyt}.
In Section 3, we focus on the way of coalescing an external branch and prove Theorem \ref{alpha-1}, Proposition \ref{dust}, Theorem \ref{size1}, {Proposition \ref{2qk} and Corollary \ref{yk}}.
Section 4 is devoted to the size of the largest block and Theorem \ref{max} and {Corollary \ref{max1} are} proved.

\section{{External branch lengths}}
\subsection{Ranked $\Lambda$-coalescent}
Assume from now on that $1<\alpha<2$. Let $\Pi=(\Pi(t), t\geq 0)$ be the Beta-coalescent {and denote by $K=(K(t),t> 0)$ the block-counting process of $\Pi$, i.e.,
$K(t)$ stands for the number of blocks of $\Pi(t)$.}
 It is known that $\Pi$ is {coming down} from infinity: for any {$t> 0$}, $K(t)$ is finite almost surely (\cite{Sch00de}).  
Recall that for any $t\geq 0$, $\Pi(t)$ is an exchangeable random partition of $\mathbb{N}$. 
Applying {Kingman's paintbox theorem} on exchangeable random partitions (\cite{Kin82a}), {almost surely, for every block $B\in \Pi(t)$, there exists the following limit which is called the asymptotic frequency of $B$:
$$\lim_{m\to \infty}\frac{1}{m}\sum_{i=1}^{m}\ind_{i\in B}.$$
 Furthermore, {when $t>0$,} the sum of all asymptotic frequencies equals 1 (\cite{Pit99}).When $t=0$, every block is a singleton and hence has the asymptotic frequency $0$.  Pitman (\cite{Pit99}) shows that {almost surely  for all $t\geq 0$}, every block in $\Pi(t)$ has the asymptotic frequency. Hence if $t>0$, one can reorder all the asymptotic frequencies in a non-increasing way to define a sequence $\Theta(t)=\{{\theta_1(t),\theta_2(t),\cdots, \theta_{K(t)} (t)}\}$ where ${\theta_1(t)\geq \theta_2(t)\geq \cdots \geq \theta_{K(t)}(t)}$ and ${\sum_{i=1}^{K(t)}\theta_i(t)=1}$. {At time $t=0$, since every block has asymptotic frequency $0$, one can naturally set $\Theta(0)=\{0,0,\cdots\}$}. Then the process $\Theta=(\Theta(t),t\geq 0)$ is well defined and called the {\it ranked $\Lambda$-coalescent}.

Given $\Theta(t)$ with $t>0$, one can recover the distribution of $\Pi(t)$ using again {Kingman's paintbox theorem}.  Let us at first divide $[0,1]$ into $K(t)$ intervals such that the lengths of intervals correspond one to one to the elements of $\Theta(t)$. Then we throw individuals $1,2,\cdots$ uniformly and independently into $[0,1]$. Finally, all individuals within one interval form a block and this procedure provides a random exchangeable partition 
which has the same law as $\Pi(t)$.
It is of course possible, thanks to the consistency property, to build the restricted partition $\Pi^{(n)}(t)$ using the same procedure 
{but  throwing nothing but $n$ particles instead of an infinity.}{This construction will be the key point of our proofs.}

\subsection{Properties of the ranked $\Lambda$-coalescent}
 Let $K(t,x):=\#\{i: \theta_i(t)\leq x\}$ for any $x\in[ 0,1]$. Let $\varsigma (t)$ be a 
size-biased picking of $\Theta(t)$, i.e., $\varsigma(t)$ is a discrete random variable such that {\begin{equation}
\label{eq:sigma}
\mathbb{P}(\varsigma (t)=\theta_i(t)|\Theta(t))=\theta_i(t)\times\#\{j: \theta_j(t)=\theta_i(t), \, 1\leq j\leq K(t)\}, \,\, 1\leq i\leq K(t).
\end{equation} 
One can construct or regard $\varsigma(t)$ in the following way: Suppose that $[0,1]$ is divided into $K(t)$ intervals whose lengths are in one-to-one correspondence to  the elements of $\Theta(t)$. We throw a particle uniformly and independently over $[0,1]$ and {$\varsigma(t)$ is} the length of the interval containing this particle.  }

Recall the measure $\mu$ defined in (\ref{eq:mulap}).
{It is easy to check that 
\begin{equation}
\label{eq:Emu}
\int_0^{\infty} y\mu(dy)=\frac{d\cl_\mu(\lambda)}{d\lambda}|_{\lambda=0}=1.
\end{equation}}
\begin{prop}\label{pr:mainprop} 
We have
\begin{equation}\label{xitx}\displaystyle \lim_{t\longrightarrow 0+}\sup_{x\geq 0}\left|\mathbb{P}(\varsigma (t)\leq t^{\frac{1}{\alpha-1}}x|\Theta(t))-\int_0^{x(\alpha\Gamma(\alpha))^{\frac{1}{\alpha-1}}}y\mu(dy)\right|=0, a.s.\end{equation}
\end{prop}
\begin{proof}
In order to simplify the notations, let us {denote} 
$f(t,x)=\mathbb{P}(\varsigma (t)\leq t^{\frac{1}{\alpha-1}}x|\Theta(t))$ and
$f(x)=\int_0^{x(
\alpha\Gamma(\alpha))^{\frac{1}{\alpha-1}}}
y\mu(dy)$.
Let 
$$\displaystyle S_t=\sup_{x\geq0 }\left|t^{\frac{1}{\alpha-1}}K\left(t, t^{\frac{1}{\alpha-1}}x\right)
-\left(\alpha\Gamma(\alpha)\right)^{\frac{1}{\alpha-1}}
\mu\left([0,
x(\alpha\Gamma(\alpha))
^{\frac{1}{\alpha-1}})\right)\right|.$$
It is shown in Theorem 1.4 of \cite{BBS08} that 
\begin{equation}\label{eq:BBS08}
\lim_{t\longrightarrow 0+}S_t=0, \quad a.s.
\end{equation}
Observe that 
 
 \begin{align}\label{decom}
f(t,x)&=\sum_{i=0}^{K(t)}\theta_i(t)
 \ind_{\{\theta_i(t)\leq t^{\frac{1}{\alpha-1}}x\}}\nonumber\\
 &=\sum_{j=0}^{n-1}\sum_{i=0}^{K(t)}\theta_i(t)
 \ind_{\{t^{\frac{1}{\alpha-1}}\frac{jx}{n}<\theta_i(t)\leq t^{\frac{1}{\alpha-1}}\frac{(j+1)x}{n}\}}.
 \end{align} 
Then 
$$f(t,x)\geq I^{(n)}_1:=\sum_{j=0}^{n-1}\sum_{i=0}^{K(t)}t^{\frac{1}{\alpha-1}}\frac{jx}{n}
 \ind_{\{t^{\frac{1}{\alpha-1}}\frac{jx}{n}<\theta_i(t)\leq t^{\frac{1}{\alpha-1}}\frac{(j+1)x}{n}\}}$$
 and
 $$f(t,x)\leq I^{(n)}_2:=\sum_{j=0}^{n-1}\sum_{i=0}^{K(t)}t^{\frac{1}{\alpha-1}}\frac{(j+1)x}{n}
 \ind_{\{t^{\frac{1}{\alpha-1}}\frac{jx}{n}<\theta_i(t)\leq t^{\frac{1}{\alpha-1}}\frac{(j+1)x}{n}\}}.$$
 For $n$ fixed and applying (\ref{eq:BBS08}), one gets for $t\to 0+$
{ $$I_1^{(n)}\stackrel{a.s.}{\longrightarrow}\sum_{j=0}^{n-1}\frac{jx}{n}\left(\alpha\Gamma(\alpha)\right)^{\frac{1}{\alpha-1}}\mu\left(\Big(\frac{jx}{n}\left(\alpha\Gamma(\alpha)\right)^{\frac{1}{\alpha-1}}, \frac{(j+1)x}{n}\left(\alpha\Gamma(\alpha)\right)^{\frac{1}{\alpha-1}})\Big]\right),$$
 and
$$I_2^{(n)}\stackrel{a.s.}{\longrightarrow}\sum_{j=0}^{n-1}\frac{(j+1)x}{n}\left(\alpha\Gamma(\alpha)\right)^{\frac{1}{\alpha-1}}\mu\left(\Big(\frac{jx}{n}\left(\alpha\Gamma(\alpha)\right)^{\frac{1}{\alpha-1}}, \frac{(j+1)x}{n}\left(\alpha\Gamma(\alpha)\right)^{\frac{1}{\alpha-1}})\Big]\right).$$}
 The above two limit values converge to $f(x)$ as $n$ goes to $\infty$. Then we can conclude.\end{proof}
 It is straightforward to see that 
 \begin{cor}\label{fcor}
 For any $f\in C_b^{0}[0, \infty)$ and $c\geq 0$, $M\in \mathbb{R}_{+}\cup \{\infty\},$
 $$\displaystyle \mathbb{E}\left[f(ct^{-\frac{1}{\alpha-1}}\varsigma (t))\ind_{\{0\leq ct^{-\frac{1}{\alpha-1}}\varsigma (t)\leq M\}}|\Theta(t)\right]\stackrel{a.s.}{\longrightarrow }\int_0^{Mc^{-1}(\alpha\Gamma(\alpha))
 ^{\frac{1}{\alpha-1}}}f\left(
 c(\alpha\Gamma(\alpha))
 ^{-\frac{1}{\alpha-1}}y\right)y\mu(dy)$$
when $t\to0+$.  \end{cor}

\subsection{External branches}
We start the proof of Theorem \ref{th:manyt} with a simpler version.
\begin{prop}\label{pr:tas}
Let $\{T_i^{(n)},{1\leq i\leq k}\}$ and $T$ be as in Theorem \ref{th:manyt}.
The following almost sure convergence holds as $n$ goes to $\infty$: 
\begin{equation}\label{equ2}
\mathbb{P}(n^{\alpha-1}T_1^{(n)}> t, n^{\alpha-1}T_2^{(n)}> t,\cdots, n^{\alpha-1}T_k^{(n)}> t|\Theta(n^{1-\alpha}t))\stackrel{a.s.}{\longrightarrow} 
\P(T>t)^k
\end{equation}
for any $t\geq 0$.
As a consequence, 
\begin{equation}\label{equ1}
n^{\alpha-1}T_1^{(n)}\stackrel{(d)}{\longrightarrow} T.
\end{equation}
\end{prop}
\begin{rem}
The convergence \eqref{equ1} has already been obtained in
\cite{DFSY} using two different methods.
\end{rem}
\begin{proof}
{For the sake of simplicity in notations, let 
$t_n=n^{1-\alpha}t$.
Let us build $\Pi^{(n)}(t)$ from $\Theta(t)$ and the paintbox construction (using $n$ particles).}
We now prove (\ref{equ2}) for $k=2$. The proof for $k> 2$ and $k=1$ follows similarly.
 Let $\bar{\varsigma}(t_n)$ be an independent copy of $\varsigma(t_n)$,
   conditionally on $\Theta(t_n)$. 
   Then, 
\begin{align}\label{t2t}
&\mathbb{P}(n^{\alpha-1}T_1^{(n)}> t, n^{\alpha-1}T_2^{(n)}> t|\Theta(t_n))\nonumber\\
=&\sum_{i,j=1, i\neq j}^{K(t_n)}\theta_i(t_n)\theta_j(t_n)\Big(1-\theta_i(t_n)-\theta_j(t_n)\Big)^{n-2}\nonumber\\
=&\sum_{i,j=1}^{K(t_n)}\theta_i(t_n)\theta_j(t_n)\Big(1-\theta_i(t_n)-\theta_j(t_n)\Big)^{n-2}- \sum_{i=1}^{K(t_n)}\theta_i(t_n)^2
\Big(1-2\theta_i(t_n)\Big)^{n-2}
\nonumber\\
=&\E[\Big(1-\varsigma(t_n)
-\bar{\varsigma}(t_n)\Big)^{n-2}|\Theta(t_n)]- \E[\varsigma(t_n)
\Big(1-2\varsigma(t_n)\Big)^{n-2}
|\Theta(t_n)],\nonumber
\end{align}
Using Corollary \ref{fcor}, the second term converges almost surely to $0$. Let $M$ be a real positive number and write the first term as 
\begin{align*}
\E[\Big(1-\varsigma(t_n)-\bar{\varsigma}(t_n)\Big)^{n-2}|\Theta(t_n)]&=I_1+I_2,\\
\end{align*}
where 
$$I_1=\E[\Big(1-\varsigma(t_n)-\bar{\varsigma}(t_n)\Big)^{n-2}\ind_{\varsigma(t_n)\leq Mn^{-1}, \bar{\varsigma}(t_n)\leq Mn^{-1}}|\Theta(t_n)],$$
$$I_2=\E[\Big(1-\varsigma(t_n)-\bar{\varsigma}(t_n)\Big)^{n-2}\ind_{\{\varsigma(t_n)\leq Mn^{-1}, \bar{\varsigma}(t_n)\leq Mn^{-1}\}^c}|\Theta(t_n)].$$
By Proposition \ref{pr:mainprop}, 
$$I_2\leq 1-\mathbb{P}(\bar{\varsigma}(t_n)\leq Mn^{-1}|\Theta(t_n)])^2\stackrel{a.s.}{\to}1-(1-\int_{Mt^{\frac{1}{1-\alpha}}}^{\infty}y\mu(dy))^2.$$
The limit value goes to $0$ as $M$ tends to $\infty.$
 For $I_1$, notice that $x\mapsto (1-n^{-1}x)^{n-2}$ converges uniformly to $x\mapsto e^{-x}$ for $0\leq x\leq 2M$ as $n$ tends to $\infty$. Then  
$$I_1-\E[\exp\left(-n\varsigma(t_n)-n\bar{\varsigma}(t_n)\right)\ind_{\varsigma(t_n)\leq Mn^{-1}, \bar{\varsigma}(t_n)\leq Mn^{-1}}|\Theta(t_n)]\stackrel{a.s.}{\to}0.$$
Now, thanks to Corollary \ref{fcor}, we get
\begin{align*}
&\E[\exp\left(-n\varsigma(t_n)-n\bar{\varsigma}(t_n)\right)\ind_{\varsigma(t_n)\leq Mn^{-1}, \bar{\varsigma}(t_n)\leq Mn^{-1}}|\Theta(t_n)]\\
=&\E[\exp\left(-n\varsigma(t_n)\right)\ind_{\varsigma(t_n)\leq Mn^{-1}}\times \exp\left(-n\bar{\varsigma}(t_n)\right)\ind_{\bar{\varsigma}(t_n)\leq Mn^{-1}}|\Theta(t_n)]\\
\stackrel{a.s.}{\to}&\left(\int_0^{M^{-1}(\frac{\alpha\Gamma(\alpha)}{t})
^{\frac{1}{\alpha-1}}}
e^{-(\frac{\alpha\Gamma(\alpha)}{t})
^{-\frac{1}{\alpha-1}}y}y\mu(dy)\right)^2\\
 \stackrel{M\to\infty}{\to}&(1+\frac{t}{\alpha\Gamma(\alpha)})^{-2\frac{\alpha}{\alpha-1}}.
\end{align*}
Then we can conclude. 
\end{proof}

\subsection{{The block-counting process in small time}}
{Recall that $K^{(n)}=(K^{(n)}(t),t> 0)$ and $K=(K(t),t> 0)$ are respectively the block-counting processes of $\Pi^{(n)}$ and $\Pi$.}
 \begin{lem}\label{suede}
 Let $t>0$ and $t_n=n^{1-\alpha}t$. We have
\begin{equation}\label{expk}
\mathbb{E}[K^{(n)}(t_n)|\Theta(t_n)]=\sum_{i=1}
^{K(t_n)}1-\Big(1-\theta_i(t_n)\Big)^n.\end{equation}
and
\begin{align}\label{vark}
\Var(K^{(n)}(t_n)|\Theta(t_n))&=\sum_{i=1}^{K(t_n)}\Big(1-\theta_i(t_n)\Big)^n\Big(1-(1-\theta_i(t_n))^n\Big)\nonumber\\
&+\sum_{i,j=1, i\neq j}^{K(t_n)}\Big(1-\theta_i(t_n)-\theta_j(t_n)\Big)^n-\Big(1-\theta_i(t_n)\Big)^n\Big(1-\theta_j(t_n)\Big)^n.
\end{align}
Furthermore, 
\begin{equation}\label{limEk}
\frac{\mathbb{E}[K^{(n)}(t_n)|\Theta(t_n)]}{n}\stackrel{a.s.}{\longrightarrow }(1+\frac{t}{\alpha\Gamma(\alpha)})
^{-\frac{1}{\alpha-1}}, \,n\to \infty.
 \end{equation}
and
 \begin{equation}\label{limVk}
\frac{Var(K^{(n)}(t_n)|\Theta(t_n))}{n}\stackrel{a.s.}{\longrightarrow }(2^{1-\alpha}+\frac{t}{\alpha\Gamma(\alpha)})^{-\frac{1}{\alpha-1}}-(1+\frac{t}{\alpha\Gamma(\alpha)})^{-\frac{1}{\alpha-1}},\,n\to \infty.
\end{equation}
\end{lem}
\begin{rem}
It can be deduced from (\ref{limEk}) and (\ref{limVk}) that 
$$\frac{K^{(n)}(t_n)
}{n(1+\frac{t}{\alpha\Gamma(\alpha)})
^{-\frac{1}{\alpha-1}}}\stackrel{P}{\to}1,$$
whereas, {interestingly,} due to Proposition \ref{pr:mainprop} or Theorem 1.1 of \cite{BBS08},
\begin{equation}\label{k}\frac{K(t_n)
}{n(\frac{t}{\alpha\Gamma(\alpha)})
^{-\frac{1}{\alpha-1}}}\stackrel{a.s.}{\to} 1.\end{equation}
\end{rem}
\begin{proof}
The equalities (\ref{expk}) and (\ref{vark}) come directly from (4.1) and (4.2) in \cite{HJ08}. The arguments to prove (\ref{limEk})
and (\ref{limVk}) include (\ref{k}) and those used in the proof of Proposition \ref{pr:mainprop}. 
To be more clear, we just show the proof of (\ref{limEk}) and leave the other to the readers.
{
\begin{align*}
 \frac{\mathbb{E}[K^{(n)}(t_n)|\Theta(t_n)]}{n}
&=
\frac{K(t_n)}{n}-
\sum_{i=0}^{K(t_n)}n^{-1}(1-\theta_i)^n\\
&\leq \frac{K(t_n)}{n}-
\sum_{j=0}^{n-1}n^{-1}(1-\frac{j+1}{n})^n
 \Big(K(t_n,\frac{j+1}{n})-
 K(t_n,\frac{j}{n})\Big).
 \end{align*}
 In the same way,
 \begin{align*}
 \frac{\mathbb{E}[K^{(n)}(t_n)|\Theta(t_n)]}{n}
&\geq \frac{K(t_n)}{n}-
\sum_{j=0}^{n-1}n^{-1}(1-\frac{j}{n})^n
 \Big(K(t_n,\frac{j+1}{n})-
 K(t_n,\frac{j}{n})\Big)\\
 \end{align*}
{While} (\ref{eq:BBS08}) shows that 
{\begin{align*}
&\sup_{0\leq j\leq n-1}\left|t_{n}^{\frac{1}{\alpha-1}}\left(K(t_n,\frac{j+1}{n})-
 K(t_n,\frac{j}{n})\right)
-\alpha\Gamma(\alpha)^{\frac{1}{\alpha-1}}\mu\Big(\Big((\frac{\alpha\Gamma(\alpha)}{t})^{\frac{1}{\alpha-1}}\frac{j}{n},(\frac{\alpha\Gamma(\alpha)}{t})^{\frac{1}{\alpha-1}}\frac{j+1}{n}\Big]\Big)\right|\\
=&\sup_{0\leq j\leq n-1}\left|n^{-1}t^{\frac{1}{\alpha-1}}\left(K(t_n,\frac{j+1}{n})-
 K(t_n,\frac{j}{n})\right)
-\alpha\Gamma(\alpha)^{\frac{1}{\alpha-1}}\mu\Big(\Big((\frac{\alpha\Gamma(\alpha)}{t})^{\frac{1}{\alpha-1}}\frac{j}{n},(\frac{\alpha\Gamma(\alpha)}{t})^{\frac{1}{\alpha-1}}\frac{j+1}{n}\Big]\Big)\right|\\
\leq &2S_{t_n}\stackrel{a.s.}{\longrightarrow}0.
\end{align*}}
Notice also that 
{$$\lim_{n\to\infty}\sum_{j=0}^{n-1}(1-\frac{j}{n})^n=\lim_{n\to\infty}\sum_{j=0}^{n-1}(1-\frac{j+1}{n})^n=\sum_{j=1}^{\infty}e^{-j}<\infty.$$}
Hence using (\ref{k}), almost surely,
{\begin{align*}
\frac{\mathbb{E}[K^{(n)}(t_n)|\Theta(t_n)]}{n} &\sim\left(\frac{\alpha\Gamma(\alpha)}{t}\right)^{\frac{1}{\alpha-1}}\left(1-
\sum_{j=0}^{n-1}(1-\frac{j}{n})^n
 \mu\Big(\Big((\frac{\alpha\Gamma(\alpha)}{t})^{\frac{1}{\alpha-1}}\frac{j}{n},(\frac{\alpha\Gamma(\alpha)}{t})^{\frac{1}{\alpha-1}}\frac{j+1}{n}{\Big]}\Big)\right)\\
 &\stackrel{n\to\infty}{\to}\left(\frac{\alpha\Gamma(\alpha)}{t}\right)^{\frac{1}{\alpha-1}}(1-
\int_0^{\infty} e^{-y\left(\frac{t}{\alpha\Gamma(\alpha)}\right)^{\frac{1}{\alpha-1}}}\mu(dy))\\
&=\left(1+\frac{t}{\alpha\Gamma(\alpha)}\right)
^{-\frac{1}{\alpha-1}}.
\end{align*}}}
 \end{proof}

We are now able to prove our first result.
\begin{proof}[Proof of Theorem \ref{th:manyt}] We will prove only the version for $k=2$. {For any $0\leq t_1\leq t_2$, write 
\begin{align*}
&\quad \mathbb{P}(n^{\alpha-1}T_1^{(n)}> t_1, n^{\alpha-1}T_2^{(n)}> t_2)\\
&=\mathbb{P}(n^{\alpha-1}T_1^{(n)}> t_1, n^{\alpha-1}T_2^{(n)}> t_1)\mathbb{P}(n^{\alpha-1}T_2^{(n)}> t_2|n^{\alpha-1}T_1^{(n)}> t_1, n^{\alpha-1}T_2^{(n)}> t_1).
\end{align*}
Proposition \ref{pr:tas} gives that the first term of the above product has limit value 
$$\mathbb{P}(T>t_1)^2
{=\left(1+\frac{t_1}{\alpha\Gamma(\alpha)}\right)
^{-\frac{2}{\alpha-1}}.}$$ Lemma \ref{suede} implies that conditional on $\{n^{\alpha-1}T_1^{(n)}> t_1, n^{\alpha-1}T_2^{(n)}> t_1\}$,  the random variable $\frac{K^{(n)}(n^{1-\alpha}t_1)}{n}$ converges in probability to $(1+\frac{t_1}{\alpha\Gamma(\alpha)})
^{-\frac{1}{\alpha-1}}$. For any $j\geq 2$, let $\tilde{T}_1^{(j)}$ be independent of $\Pi^{(n)}$ and have the same law as $T_1^{(j)}$. Using the Markov property of $\Pi^{(n)}$, one obtains
\begin{eqnarray*}
&&\mathbb{P}(n^{\alpha-1}T_2^{(n)}> t_2|n^{\alpha-1}T_1^{(n)}> t_1, n^{\alpha-1}T_2^{(n)}> t_1)\\
&=&\P(n^{\alpha-1}\tilde{T}_1^{(K^{(n)}(n^{1-\alpha}t_1))}> t_2-t_1|n^{\alpha-1}T_1^{(n)}> t_1, n^{\alpha-1}T_2^{(n)}> t_1)\\
&\stackrel{P}{\to}&\P\left(T>(t_2-t_1)
{\left(1+\frac{t_1}{\alpha\Gamma(\alpha)}\right)
^{-1}}
\right)
\\
&=&\left(1+\frac{t_1}{\alpha\Gamma(\alpha)}\right)
^{\frac{1}{\alpha-1}}\left(1+\frac{t_2}{\alpha\Gamma(\alpha)}\right)
^{-\frac{1}{\alpha-1}}
\end{eqnarray*}
when $n$ tends to $\infty$.
Then we can conclude.}
\end{proof}

\section{{The way of coalescing an external branch}}
\subsection{{The size of the jump}}
Let us look at the random variable $Q^{(n)}$.
\begin{proof}[Proof of Theorem \ref{alpha-1}]

Assume that at some time $t$, $K^{(n)}(t)=b$ and $\{1\}\in \Pi^{(n)}(t)$. The coalescence of $\{1\}$ with some other $k-1$ blocks happens at rate  
$$\lambda_{1,b,k}:=\int_0^1{b-1\choose k-1}x^{k}(1-x)^{b-k}x^{-2}\Lambda(dx)=\frac{\Gamma(k-\alpha)\Gamma(b-k+\alpha)}{\Gamma(\alpha)\Gamma(2-\alpha)\Gamma(k)\Gamma(b-k+1)}.$$
The total rate {at which the singleton $\{1\}$ participates in} a coalescence event is
\begin{align}\label{2q}
g_{1,b}:&=\int_0^1\sum_{k=2}^b{b-1\choose k-1}x^{k}(1-x)^{b-k}x^{-2}\Lambda(dx)\nonumber\\
&=\int_0^1(1-(1-x)^{b-1})x^{-1}\Lambda(dx)\nonumber\\
&=\int_0^1(b-1)(1-t)^{b-2}\rho_1(t)dt,
\end{align}
where 
$$\rho_1(t)=\int_t^1x^{-1}\Lambda(dx)\sim \frac{t^{1-\alpha}}{(\alpha-1)\Gamma(\alpha)\Gamma(2-\alpha)}$$ when $t$ tends to $0+$. We get, thanks to Stirling's formula,
$$g_{1,b}\sim  \frac{b^{\alpha-1}}{(\alpha-1)\Gamma(\alpha)}$$
when $b$ tends to $\infty.$
If the next coalescence after $t$ involves $\{1\}$, then using the strong Markov property of $\Pi^{(n)}$, the probability for $\{1\}$ to coalesce with some other $k-1$ blocks is 
\begin{equation}\label{kq}\frac{\lambda_{1,b,k}}{g_{1,b}}\sim q_k=\frac{\Gamma(k-\alpha)(\alpha-1)}{\Gamma(k)\Gamma(2-\alpha)}\end{equation}
when $b$ tends to $\infty.$
In this way, if we know the value $K^{(n)}((T_1^{(n)})_-)$, then we can obtain the probability for $\{1\}$ to coalesce with $k-1$ blocks. {Notice that $\frac{K^{(n)}((T_1^{(n)})_-)}{n}$ converges in distribution to $(1+\frac{T}{\alpha\Gamma(\alpha)})
^{-\frac{1}{\alpha-1}}$ (see Corollary 5.3 of \cite{DFSY} or implicitly from Theorem \ref{th:manyt} and Lemma \ref{suede}), one can get the following, due to (\ref{kq}),
\begin{equation}\label{qcon}\mathbb{P}(Q^{(n)}=k)=\mathbb{E}[\frac{\lambda_{1,K^{(n)}((T_1^{(n)})_-),k}}{g_{1,K^{(n)}((T_1^{(n)})_-)}}]\longrightarrow q_k\end{equation} 
when $n$ tends to $\infty.$}

{The asymptotic independence of $T_1^{(n)}$ and $Q^{(n)}$ is clear, since $Q^{(n)}$ only depends on $K^{(n)}((T_1^{(n)})_-)$ which tends to $\infty$ in probability when $n$ goes to $\infty$. Then we can conclude.}
\end{proof}
\begin{rem}
Following the same arguments, Theorem \ref{alpha-1} is still valid for the more general class of coalescents
{satisfying the following condition when $t$ tends to 0:
$$\int_t^1x^{-2}\Lambda(dx)\sim Ct^{-\alpha}, C>0.$$}
\end{rem}

\begin{rem}\label{bsinf}
{We can use similar arguments in the Bolthausen-Sznitman case to get that $\mathbb{P}(Q^{(n)}=k)\to0$ for any $k\in\N$. 
The result actually remains true for the more general class where
$$\int_t^1x^{-2}\Lambda(dx)\sim Ct^{-1}, C>0.$$}
\end{rem}
\begin{proof}[{Proof of Proposition \ref{2qk}}]
1) Recall that $q_k=\frac{\Gamma(k-\alpha)(\alpha-1)}{\Gamma(k)\Gamma(2-\alpha)}$. It then suffices to prove that 
$$\lim_{t\to0+}\frac{\mathbb{P}(\beta(t)=k)}{t}=\frac{\Gamma(k-\alpha)}{\Gamma(k)\Gamma(\alpha)\Gamma(2-\alpha)}.$$
To simplify the notations, let $t_{\alpha}=\left(\frac{t}{\alpha\Gamma(\alpha)}\right)^\frac{1}{\alpha-1}$ and $\rho_x=\mu([x,\infty))$. Recall (\ref{eq:beta}) and let $0<\eta<\frac{2-\alpha}{2}$,  then for any $k\geq 2$, 
{\begin{align*}
\mathbb{P}(\beta(t)=k)&=\frac{t_{\alpha}^{k-1}}{\Gamma(k)}\int_0^{\infty}e^{-xt_{\alpha}}x^k\mu(dx)\nonumber\\
&=\frac{t_{\alpha}^{k-1}}{\Gamma(k)}\int_0^{\infty}\rho_{x}e^{-xt_{\alpha}}x^{k-1}(k-xt_{\alpha})dx\nonumber\\
&=I_1+I_2,
\end{align*}
where 
$$I_1=\frac{t_{\alpha}^{k-1}}{\Gamma(k)}\int_0^{t_{\alpha}^{-\eta}}\rho_{x}e^{-x t_{\alpha}}x^{k-1}(k-xt_{\alpha})dx,$$
$$I_2=\frac{t_{\alpha}^{k-1}}{\Gamma(k)}\int_{t_{\alpha}^{-\eta}}^{\infty}\rho_{x}e^{-x t_{\alpha}}x^{k-1}(k-xt_{\alpha})dx.$$
For $t$ small enough, it is easy to get $I_1\leq \frac{k}{\Gamma(k)}t_{\alpha}^{k(1-\eta)-1}=o(t)$.
 To deal with $I_2$, recall from Equation (33) of \cite{BBS08} that
 \begin{equation}\label{eq:aueuemu}\rho_x=\mu([x,\infty))\sim\frac{x^{-\alpha}}{\Gamma(2-\alpha)}\end{equation}
 when $x$ goes to $\infty$.
Notice that $t_{\alpha}^{-\eta}$ goes to $\infty$ when $t$ tends to $0+$. Let $0<\varepsilon<1$, then for $t$ small enough, we have 
$$1-\varepsilon<\frac{\rho_{x}}{x^{-\alpha}/\Gamma(2-\alpha)}\leq 1+\varepsilon,\,\,\text{for all} \,x\geq t_{\alpha}^{-\eta}.$$
Since $\varepsilon$ can be arbitrarily small, using a change of variable $y=xt_{\alpha}$, one gets 
$$t_{\alpha}^{{1-\alpha}}I_2\to \frac{1}{\Gamma(k){\Gamma(2-\alpha)}}\int_0^{{\infty}} e^{-y}y^{k-1-\alpha}(k-y)dy=\frac{\alpha\Gamma(k-\alpha)}{\Gamma(k){\Gamma(2-\alpha)}}$$
when $t\to 0+$.}
Then we can obtain (\ref{qk=}).

2) A simple calculation shows that $\frac{\mathbb{P}(\beta(t)\geq 2)}{t}$ converges to $\frac{1}{(\alpha-1)\Gamma(\alpha)}$ when $t$ tends to $0+$. Hence the first equality of (\ref{lq}) holds, using the dominated convergence theorem. For the second equality, the formulas (\ref{eq:beta}) and (\ref{eq:mulap}) imply that
\begin{align}\label{lbt}
\mathbb{E}[e^{-\lambda \beta(t)}]&=\sum_{k\geq 1}e^{-\lambda k}\mathbb{P}(\beta(t)=k)\nonumber\\
&=e^{-\lambda}\int_0^{\infty}e^{(e^{-\lambda}-1)xt_{\alpha}}x\mu(dx)\nonumber\\
&=e^{-\lambda}\left(1+(t_{\alpha}(1-e^{-\lambda}))^{\alpha-1}\right)^{\frac{\alpha}{1-\alpha}}.
\end{align}
Meanwhile, using the same arguments
$$\mathbb{P}(\beta(t)=1)=(1+t_{\alpha}^{\alpha-1})^{\frac{\alpha}{1-\alpha}}.$$
Then we can obtain (\ref{lq}).\end{proof}
Let us consider the case of coalescents satisfying $\int_0^1x^{-1}\Lambda(dx)<\infty$.
\begin{proof}[Proof of Proposition \ref{dust}]
In this case, the process $\Pi^{(n)}$ can be constructed using a subordinator. 
This construction can be found {on} page $7$ of \cite{GIM08} and the original idea is in \cite{Pit99}. 
Let $\nu(dx)=x^{-2}\Lambda(dx)$ and $\tilde{\nu}$ be the 
{push-forward} of $\nu$ by the transformation $x\to -\ln(1-x)$. Let $(\tilde{S}_t, t\geq 0)$ be a subordinator with L\'evy measure $\tilde{\nu}$ and $S_t=e^{-\tilde{S}_t}$. 
Then $(S_t, t\geq 0)$ is a non-increasing positive pure-jump process with $S_0{=1}$. 
Put individuals $1,2,\cdots, n$ uniformly and independently over $(0,1]$. 
Let $t_1$ be the first time when $(S_{t_1}, S_{t_1-}]$ contains at least one individual, then we set $\Pi^{(n)}(s)=\{{\{1\},\{2\},\cdots, \{n\}}\}$ for $0\leq s<t_1$. 
We regroup the individuals located in $(S_{t_1}, S_{t_1-}]$ into one block and let $\Pi^{(n)}(t_1)$ be the set of this block and the rest singletons. 
The block is then put uniformly and independently into $(0, S_{t_1}]$. 
Find the next time $t_2$, such that $(S_{t_2}, S_{t_2-}]$ contains at least one block (singleton or not) and we regroup all blocks in this interval into one bigger block which will be again put independently into $(0,S_{t_2}]$. 
In general, at each time $t$, $\Pi^{(n)}(t)$ is the set of the blocks located in $(0,S_t]$ and also the union of blocks located in $(S_t,S_{t-}]$. 
This operation can be iterated until reaching the MRCA and the process resulted has the same law as $\Pi^{(n)}$.

Notice that this construction is consistent, i.e., if we add a $n+1$-th individual to get $\Pi^{(n+1)}$, the structure of $\Pi^{(n)}$ is conserved.  
To see how many blocks will merge with $\{1\}$ in the limit, we assume that $\{1\}$ is put into $(S_{t_1},S_{t_1-}]$ at a certain time. 
As $n$ goes to $\infty$, the number of singletons put into the same interval also goes to $\infty$ with probability $1$. 
According to the construction, all singletons put for the first time into the same interval will be coalesced together. Hence {$Q^{(n)}$ converges in probability to infinity}.
\end{proof}

\subsection{Minimal clade size}

Let $(s_i^{(n)}(t),{1\leq i\leq K^{(n)}(t)})$ be the increasing sequence of the smallest elements of blocks of $\Pi^{(n)}(t)$.
We have the following lemma.

\begin{lem}\label{ns}
For any $t>0$ and $k\in \mathbb{N}$, let $t_n=n^{1-\alpha}t$ and
define the event $E_{n,k}=
\{s_{1}^{(n)}(t_n)=1,\cdots, s_{k}^{(n)}(t_n)=k\}$. Then
$$\mathbb{P}(E_{n,k}|\Theta(t_n))\stackrel{a.s.}{\longrightarrow }1$$
as $n$ tends to $\infty.$
\end{lem}

\begin{proof}
{}
We only  need to prove that the probability for individuals $1$ and $2$ to be in the same block {of} $\Pi^{(n)}(t_n)$ tends to $0$. The case of $k\geq 3$ follows in the same way. 
Let us write this event $\{1\overset{t_n}{\sim}2\}$.
Then
\begin{align*}
\mathbb{P}(1\overset{t_n}{\sim}2|\Theta(t_n))
&=\sum_{i=1}^{K({t_n})}\theta_i(t_n)^2\\
&=\mathbb{E}[\varsigma(t_n)|\Theta(t_n)]\stackrel{a.s.}{\longrightarrow }0,
\end{align*}
where the convergence is due to Corollary \ref{fcor}.
\end{proof}

\begin{theo}\label{size1t}
Let $t>0$ and $t_n=n^{1-\alpha}t$. For ${1\leq i\leq K^{(n)}}(t)$, let $K_i^{(n)}(t)$ be the size of the block containing $s_i^{(n)}(t)$. Then for any $k\in\mathbb{N}$ and $(r_1,\dots,r_k)\in\N^k$,
as $n\to\infty$,
{\begin{equation} \mathbb{P}(K_1^{(n)}(t_n)=r_1, \cdots, K_k^{(n)}(t_n)=r_k |\Theta(t_n))\stackrel{a.s.}{\longrightarrow} \prod_{i=1}^{k}\P(\beta(t)=r_i)\end{equation}
where $\beta(t)$ is defined in \eqref{eq:beta}.}
\end{theo}

\begin{proof}{Let $t_n=n^{1-\alpha}t$.}
Define the event $E_{n,r}=\{K_1^{(n)}(t_n)=r_1, \cdots, K_k^{(n)}(t_n)=r_k\}$ and recall $E_{n,k}$ defined in Lemma \ref{ns}.  Let $(\varsigma_i (t_n),{1\leq i\leq k})$ be $k$ independent copies of $\varsigma(t_n)$ which is
defined in \eqref{eq:sigma},
conditional on $\Theta(t_n)$. For $1\leq i\leq k$, $\varsigma_i(t_n)$ denotes the size of the subinterval into which $i$ is thrown 
{in the paintbox construction of $\Pi^{(n)}(t_n)$ with $n\geq k$}. 
Due to Lemma \ref{ns}, for $n$ large enough we can almost surely {approach}  $\mathbb{P}(E_{n,r}|\Theta(t_n))$ by

$$\mathbb{P}(E_{n,r}|E_{n,k},\Theta(t_n))
=\mathbb{E}[{n-k\choose r_1-1,\cdots, r_k-1}\Pi_{j=1}^{k}(\varsigma_j(t_n))^{r_j-1}(1-\sum_{j=1}^{k}\varsigma_j(t_n))^{n-\sum_{j=1}^{k}r_j}|E_{n,k}, \Theta(t_n)].
$$
The Lemma \ref{ns} implies again that the difference
\begin{align*}
&\mathbb{E}[{n-k\choose r_1-1,\cdots, r_k-1}\Pi_{j=1}^{k}(\varsigma_j(t_n))^{r_j-1}(1-\sum_{j=1}^{k}\varsigma_j(t_n))^{n-\sum_{j=1}^{k}r_j}|E_{n,k}, \Theta(t_n)]\\
&-\mathbb{E}[{n-k\choose r_1-1,\cdots, r_k-1}\Pi_{j=1}^{k}(\varsigma_j(t_n))^{r_j-1}(1-\sum_{j=1}^{k}\varsigma_j(t_n))^{n-\sum_{j=1}^{k}r_j}|\Theta(t_n)]
\end{align*}
converges almost surely to $0$.
We then obtain

\begin{align*}
&\mathbb{E}[{n-k\choose r_1-1,\cdots, r_k-1}\Pi_{j=1}^{k}(\varsigma_j(t_n))^{r_j-1}(1-\sum_{j=1}^{k}\varsigma_j(t_n))^{n-\sum_{j=1}^{k}r_j}|\Theta(t_n)]\\
=&\mathbb{E}[\Pi_{j=1}^{k}\frac{1}{(r_j-1)!}(n\varsigma_j(t_n))^{r_j-1}e^{-n\varsigma_j(t_n)}|\Theta(t_n)]+o_{a.s.}(1) ,
\end{align*}
where $o_{a.s.}(1)$ is a term converging almost surely to $0$
when $n$ tends to $\infty$.  
The result thus follows from Corollary \ref{fcor}.
\end{proof}

We next show how the external branch of $1$ is connected to the whole process. Let $\Pi^{(2,n)}$ be the restriction of $\Pi^{(n)}$ from $\mathbb{N}_n$ to $\{2,3,\cdots, n\}$. By consistency and exchangeability of $\Pi^{(n)}$, $\Pi^{(2,n)}$ has the same law as $\Pi^{(n-1)}$ except for the integer notations. Given $\Pi^{(2,n)}$, one can attach $\{1\}$ to $\Pi^{(2,n)}$ following the \textit{recursive construction} introduced in \cite{DFSY}. One thing important is that $n^{\alpha-1}T_1^{(n)}$ and $\Pi^{(2,n)}$ are asymptotically independent. The following lemma is given in the proof of Theorem 5.2 of \cite{DFSY}.

\begin{lem}\label{proba}
Let $t\geq 0$. As $n$ tends to $\infty$,
\begin{align}\mathbb{P}(n^{\alpha-1}T_1^{(n)}\geq t |\Pi^{(n,2)})&\stackrel{P}{\longrightarrow} (1+\frac{t}{\alpha\Gamma(\alpha)})^{-\frac{\alpha}{\alpha-1}}.\end{align}
\end{lem}

Now we are able to deal with the minimal clade size. 
\begin{proof}[Proof of Theorem \ref{size1}]
Recall that $Q^{(n)}$ is the number of blocks involved in the coalescence of $\{1\}$.  Then $Y^{(n)}$ is just the sum of the $Q^{(n)}$ block sizes (one of these blocks is $\{1\}$). It suffices to determine the size of each block involved in the coalescence event.  By exchangeability of the coalescent, the $Q^{(n)}-1$ blocks {not being $\{1\}$} will be chosen randomly at time $(T_1^{(n)})_-$. Hence by the strong Markov property of $\Pi^{(n)}$, the joint distribution of the sizes of the randomly chosen $Q^{(n)}-1$ blocks has the same law as the distribution of  $(K_2^{(n)}((T_1^{(n)})_-), \cdots, K_{Q^{(n)}}^{(n)}((T_1^{(n)})_-))$. Hence
\begin{equation*}\label{yn1}Y^{(n)}\overset{(d)}{=}1+\sum_{i=2}^{Q^{(n)}}K_i^{(n)}((T_1^{(n)})_-).\end{equation*}
Let $K^{(2,n)}=(K^{(2,n)}(t), t\geq 0)$ be the block-counting process of $\Pi^{(2,n)}$ and $(K_1^{(2,n)}(t),K_2^{(2,n)}(t),\cdots, K_{K^{(2,n)}(t)}^{(2,n)}(t))$ be the {vector of the block sizes of $\Pi^{(2,n)}(t)$, increasingly ordered by their least elements}. Notice that if $t<T_1^{(n)}$, $K_i^{(n)}(t)=K_{i-1}^{(2,n)}(t)$ for $1\leq i\leq K^{(n)}(t)$. Therefore
\begin{equation}\label{yn1}Y^{(n)}\overset{(d)}{=}1+\sum_{i=1}^{Q^{(n)}-1}K_{i}^{(2,n)}((T_1^{(n)})_-).\end{equation}
The formula (\ref{qcon}) shows that the law of $Q^{(n)}$ is uniquely determined by $K^{(2,n)}((T_1^{(n)})_-)$. 
As long as $K^{(2,n)}((T_1^{(n)})_-)$ goes to $\infty$, $Q^{(n)}$ converges in law to a distribution which depends only on $\alpha.$ 
While (\ref{limEk}) and (\ref{equ1}) imply that the variable $K^{(2,n)}((T_1^{(n)})_-)$ goes to $\infty$ with probability $1$. Hence $Q^{(n)}$ is asymptotically independent of $(T_1^{(n)}, K^{(2,n)}({(T_1^{(n)})_-})).$
Furthermore, Lemma \ref{proba} gives that 
$T_1^{(n)}$ and $K^{(2,n)}({(T_1^{(n)})_-}) $ are asymptotically independent. In total, $Q^{(n)}$, $T_1^{(n)}$ and $(K^{(2,n)}({(T_1^{(n)})_-})) $ are all asymptotically independent.
In the limit, using Theorem \ref{size1t},
$$Y^{(n)}\stackrel{(d)}{\longrightarrow}Y\stackrel{(d)}{=}1+\sum_{i=1}^{Q-1}{\beta}_i(T),$$
where $Q, T, (\beta_i(t))_{i\in\mathbb{N}}$ are all independent and follow respectively the limit laws of $Q^{(n)}, T_1^{(n)}, (K_{i}^{(2,n)}(t))_{i\in\mathbb{N}}$ for fixed $t\geq 0$. Then we can conclude. \end{proof}

\begin{proof}[{Proof of Corollary \ref{yk}}]
Consider the Laplace transform of $Y$. For any $\lambda>0$, using (\ref{lbt})

\begin{align*}
\mathbb{E}[e^{-\lambda Y}]&=e^{-\lambda}\mathbb{E}[(\mathbb{E}[e^{-\lambda\beta(T)}])^{Q-1}]\\
&=e^{-\lambda}\mathbb{E}\Big[\left(e^{-\lambda}\left(1+(T_{\alpha}(1-e^{-\lambda}))^{\alpha-1}\right)^{\frac{\alpha}{1-\alpha}}\right)^{Q-1}\Big]
\end{align*}
where $T_{\alpha}=\left(\frac{T}{\alpha\Gamma(\alpha)}\right)^\frac{1}{\alpha-1}.$ Denote $\Delta:=e^{-\lambda}(1+(T_{\alpha}(1-e^{-\lambda}))^{\alpha-1})^{\frac{\alpha}{1-\alpha}}$. 
Using (\ref{lq}), one gets

\begin{align*}\mathbb{E}[e^{-\lambda Y}]&=\mathbb{E}[e^{-\lambda}(1-(1-\Delta)^{\alpha-1})]\\
&=I_1+I_2,\end{align*}
where $I_1=\mathbb{E}[e^{-\lambda Y}\ind_{T_{\alpha}>\lambda^{-\frac{1}{2}}}],\,\,I_2=\mathbb{E}[e^{-\lambda Y}\ind_{T_{\alpha}\leq\lambda^{-\frac{1}{2}}}]$. The density (\ref{ft}) of $T$ implies, when $\lambda\to0+$

\begin{equation}\label{i1}I_1=O(\lambda^{\frac{\alpha}{2}})=o(\lambda^{(\alpha-1)^2}).
\end{equation}
Notice that there exists $C_{\ref{i2}}>0$ such that for any $0<\varepsilon<1$, if $\lambda$ is small enough, we have 

\begin{equation}\label{i2}|\Delta\ind_{T_{\alpha}\leq \lambda^{-\frac{1}{2}}}-(1+\frac{\alpha}{1-\alpha}(T_{\alpha}\lambda)^{\alpha-1})\ind_{T_{\alpha}\leq \lambda^{-\frac{1}{2}}}|\leq \varepsilon(T_{\alpha}\lambda)^{\alpha-1}\ind_{T_{\alpha}\leq \lambda^{-\frac{1}{2}}}+C_{\ref{i2}}\lambda. \end{equation}
Letting $\lambda\to 0+$ and using (\ref{i1}), (\ref{i2}), one obtains
$$\mathbb{E}[e^{-\lambda Y}]=1-(\frac{\alpha}{\alpha-1})^{\alpha-1}\lambda^{(\alpha-1)^2}\mathbb{E}[T_{\alpha}^{(\alpha-1)^2}]+o(\lambda^{(\alpha-1)^2}).$$
Thanks to Lemma 5.4 of \cite{BBS08} or Theorem 8.1.6 of \cite{B89}, we get 

$$\mathbb{P}(Y>k)\sim\frac{(\frac{\alpha}{\alpha-1})^{\alpha-1}\mathbb{E}[T_{\alpha}^{(\alpha-1)^2}]}{\Gamma(1-(\alpha-1)^2)}k^{-(\alpha-1)^2}=\frac{\int_{0}^{\infty}t^{{\alpha-1}}f_T(t)dt}{((\alpha-1)\Gamma(\alpha))^{\alpha-1}\Gamma(1-(\alpha-1)^2)}k^{-(\alpha-1)^2}$$
when $k\to\infty.$
\end{proof}
\section{The largest block}

In this section, we aim to prove Theorem \ref{max} and Corollary \ref{max1}.
We start with a technical lemma.
{\begin{lem}\label{lemtec}
Let $k>0$ and $X$ be a random variable distributed according to $\mu$.  Define $\mathcal{X}$ such that conditional on $X$, $\mathcal{X}$ is a Poisson variable with parameter $\frac{X}{k}$. Then for any {$x> 0,$}
$$\lim_{n\to\infty}n\mathbb{P}(\mathcal{X}\geq xn^{\frac{1}{\alpha}})=\frac{(kx)^{-\alpha}}{\Gamma(2-\alpha)}.$$ 
\end{lem}
\begin{proof}
First of all, let us consider two technical results. Let {$M=\lfloor xn^{\frac{1}{\alpha}}\rfloor$}.

$1)$ Using Stirling's formula for $M!$ and a change of variable, we get that for any $0<\beta<1$, 
\begin{align}\label{alpha1}
\int_0^{M\beta}e^{-t}\frac{t^{M}}{M!}dt&=\int_0^{M\beta}e^{M-t}\left(\frac{t}{M}\right)^M(2\pi M)^{-\frac{1}{2}}(1+O(M^{-1}))dt\nonumber\\
&=\int_0^{\beta}e^{M(1-t+\ln t)}{(\frac{M}{2\pi})^{\frac{1}{2}}}(1+O(M^{-1}))dt\nonumber\\
&= O( e^{M(1-\beta+\ln \beta)}M^{\frac{1}{2}}).
\end{align}
The last equality is due to the fact that $1-t+\ln t$ is negative and increasing for $t\in(0,1)$.

$2)$ If $\beta>1$, then
\begin{align*}
\int_{M\beta}^{\infty}e^{-t}\frac{t^{M}}{M!}dt
&=\int_{\beta}^{\infty}e^{M(1-t+\ln t)}{(\frac{M}{2\pi})^{\frac{1}{2}}}(1+O(M^{-1}))dt.
\end{align*}
Notice that $1-t+\ln t$ is strictly decreasing and concave over $[\beta, \infty]$, then there exists a positive number $\varepsilon$ such that $1-t+\ln t\leq -\varepsilon t$ for any $t\geq \beta$. Therefore, 
\begin{equation}\label{alpha2}\int_{M\beta}^{\infty}e^{-t}\frac{t^{M}}{M!}dt\leq \int_{\beta}^{\infty}e^{-\varepsilon Mt}{(\frac{M}{2\pi})^{1/2}}(1+O(M^{-1}))dt=O(e^{-\varepsilon M\beta}M^{{-1/2}}).\end{equation}
Now we can turn to the study of $\mathcal{X}.$ Thanks to successive integrations by parts,
\begin{equation}\label{m+1}
\mathbb{P}(\mathcal{X}\geq M+1)=\mathbb{E}[\int_0^{\frac{X}{k}}e^{-t}\frac{t^{M}}{M!}dt].\end{equation}
Let $0< \beta_1<1$ and $\beta_2>1$, then we have 
$$
\P(\mathcal{X}\geq M+1)
=I_1+I_2+I_3,
$$
where 
$$I_1=\mathbb{E}[\int_0^{\frac{X}{k}}e^{-t}\frac{t^{M}}{M!}dt\ind_{\{X< kM\beta_1\}}],$$ 
$$I_2=\mathbb{E}[\int_0^{\frac{X}{k}}e^{-t}\frac{t^{M}}{M!}{dt}\ind_{\{kM\beta_1\leq X\leq kM\beta_2\}}],$$ 
$$I_3=\mathbb{E}[\int_0^{\frac{X}{k}}e^{-t}\frac{t^{M}}{M!}dt\ind_{\{X>kM\beta_2\}}].$$
Now let $n$ tend to infinity.
By (\ref{alpha1}), we get
\begin{equation}\label{partie1}0\leq nI_1\leq n\mathbb{P}(X<kM\beta_1)\int_0^{M\beta_1}e^{-t}\frac{t^M}{M!}dt\longrightarrow 0
.
\end{equation}
 It is easy to verify that $\int_0^{\infty}e^{-t}\frac{t^M}{M!}dt=1$ for any integer $M\geq 0$. Then using together \eqref{eq:aueuemu} and (\ref{alpha2}), we obtain
 \begin{equation}\label{partie2}\lim_{n\to\infty} nI_3= {\lim_{n\to\infty}n}\mathbb{P}(X>kM\beta_2)=\frac{(kx\beta_2)^{-\alpha}}{\Gamma(2-\alpha)}.\end{equation}
 In the same way, we have
 \begin{equation}\label{partie3}0\leq nI_2\leq {n}\mathbb{P}(kM\beta_1\leq X\leq kM\beta_2)\longrightarrow {\frac{(kx\beta_1)^{-\alpha}}{\Gamma(2-\alpha)}-\frac{(kx\beta_2)^{-\alpha}}{\Gamma(2-\alpha)}},\,\,n\to \infty.\end{equation}
 If $\beta_1$ and $\beta_2$ are close enough to $1$, $nI_2$ can be bounded by an arbitrarily small positive number for $n$ large enough. 
 Combining (\ref{partie1}), (\ref{partie2}) and (\ref{partie3}), we conclude this lemma.
\end{proof}

To prove Theorem \ref{max}, we will use classical relations between  Beta-coalescents and continuous-state branching processes (CSBPs) developed in \cite{les705} (see also Section 2 of \cite{BBS08}).
We give a short summary to provide a minimal set of tools. A continuous-state branching process $(Z(t), t\geq 0)$ is a $[0,\infty]$-valued Markov process (in continuous time) whose transition functions $p_t(x,\cdot)$ satisfy the branching property
$$p_t(x+y,\cdot)=p_t(x,\cdot)*p_t(y,\cdot), \quad\text{for all } x,y\geq 0.$$
For each $t\geq 0$, there exists a function $u_t:[0,\infty)\to \mathbb{R}$ such that 
\begin{equation}\label{exp}\mathbb{E}[e^{-\lambda Z(t)}|Z(0)=a]=e^{-au_t(\lambda)}.\end{equation}
If almost surely, the process has no instantaneous jump to infinity, the function $u_t$ satisfies the following differential equation
$$\frac{\partial u_t(\lambda)}{\partial t}=-\Psi(u_t(\lambda)),$$
where $\Psi:[0, \infty)\longrightarrow \mathbb{R}$ is a function of the form 
$$\Psi(u)={\gamma} u+\beta u^2+\int_0^{\infty}(e^{-xu}-1+xu\ind_{\{x\leq 1\}})\pi(dx),$$
where ${\gamma}\in \mathbb{R}, \beta\geq 0$ and $\pi$ is a L\'evy measure on $(0, \infty)$ satisfying $\int_0^{\infty}(1\wedge x^2)\pi(dx)<\infty.$ The function $\Psi$ is called the branching mechanism of the CSBP.

As explained in \cite{BLG00}, a CSBP can be extended to a two-parameter random process $(Z(t,a), t\geq 0, a\geq 0)$ with $Z(0,a)=a.$ For fixed $t$, $(Z(t,a), a\geq 0)$ turns out to be a subordinator with Laplace exponent $\lambda\mapsto u_t(\lambda)$ thanks to (\ref{exp}). 

There exists a measure-valued process $(M_t, t\geq 0)$ taking values in the set of finite measures on $[0,1]$ which characterizes $(Z(t,a),t\geq 0, 0\leq a\leq 1)$. More precisely, $(M_t([0,a]), t\geq 0, 0\leq a\leq 1)$ has the same finite-dimensional distributions as $(Z(t,a), t\geq 0, 0\leq a\leq 1)$. Hence $(M_t([0,a]), 0\leq a\leq 1)$ is a subordinator with Laplace exponent $\lambda\mapsto u_t(\lambda)$  and $Z(t)=M_t([0,1])$ is a CSBP with branching mechanism $\Psi$ started at $M_0([0,1])=1$. In particular, if the branching mechanism is $\Psi(\lambda)=\lambda^{\alpha}$ and hence L\'evy measure is given by $\pi(dx)=\frac{\alpha(\alpha-1)}{\Gamma(2-\alpha)}x^{-1-\alpha}dx$, for all $t>0,$ $M_{t}$ consists only of finite number of atoms. For the construction of $(M_t([0,a]), t\geq 0, 0\leq a\leq 1)$, we refer to \cite{BBS07, les705, DK99}.

A deep relation has been revealed in \cite{les705} between the Beta-coalescent and the CSBP with branching mechanism $\Psi(\lambda)=\lambda^{\alpha}$ and L\'evy measure $\pi(dx)=\frac{\alpha(\alpha-1)}{\Gamma(2-\alpha)}x^{-1-\alpha}dx$.  The relationship is described by the following two lemmas which are respectively Lemma 2.1 and 2.2 of \cite{BBS08} and will be important in the sequel.

To save notations, from now on, $(Z(t),t\geq 0)$ always denotes a continuous state branching process  $(Z(t,1),t\geq 0)$. 
\begin{lem}\label{cle1}
Assume $(Z(t), t\geq 0)$ is a CSBP with branching mechanism
$\Psi(\lambda)=\lambda^{\alpha}$ and let
$(M_t, t\geq 0)$ be its associated measure-valued process. 
If $(\Pi(t), t\geq 0)$ is a Beta-coalescent and $(\Theta(t),t\geq 0)$ is the associated ranked coalescent, then for all $t>0$, the distribution of $\Theta(t)$ is the same as the distribution of the sizes of the atoms of the measure $\frac{M_{R^{-1}(t)}}{Z(R^{-1}(t))}$, ranked in decreasing order. 
Here $R(t)=(\alpha-1)\alpha\Gamma(\alpha)\int_0^tZ(s)^{1-\alpha}ds$ and $R^{-1}(t)=\inf\{s:R(s)>t\}$.
\end{lem}

\begin{lem}\label{cle2}

Assume $\Psi(\lambda)=\lambda^{\alpha}$. For any $t\geq0$, let $D(t) $ be the number of atoms of $M_t$, and let $J(t)=(J_1(t),\cdots, J_{D(t)}(t))$ be the sizes of the atoms of $M_t$, ranked in decreasing order. Then $D(t)$ is Poisson with mean $\gamma_t=\left((\alpha-1)t\right)^{-\frac{1}{\alpha-1}}$. 
Moreover, conditional on $D(t)=k$, the distribution of $J(t)$ is the same as the distribution of $(\gamma_t^{-1}X_1, \cdots, \gamma_t^{-1}X_k)$ where $X_1,\cdots, X_k$ are obtained by picking $k$ i.i.d. random variables with distribution $\mu$ and then ranking them in decreasing order.
\end{lem}
\begin{rem}\label{save1}
From the relation between $(M_t,t\geq 0)$ and $(Z(t,a), t\geq0, 0\leq a\leq 1)$ and also the fact that for all $t>0,$ $M_{t}$ consists only of finite number of atoms (the number is actually $D(t)$), for a given $t>0$, there exist $0\leq a_1,\cdots,a_{D(t)}\leq1 $ such that $\{Z(t,a_1)-Z(t,a_1-), \cdots, Z(t,a_{D(t)})-Z(t,a_{D(t)}-)\}$ are exactly the values of the atoms of $M_{t}$. By the strong Markov property of $(Z(t,a), t\geq0, 0\leq a\leq 1)$, for $s\geq t$, the jumping at $s$ can only happen at the points $\{(s,a_1), \cdots, (s,a_{D(t)})\}$.  Therefore, $D(t)$ decreases on $t$. 
\end{rem}
The idea of the proof of Theorem \ref{max} is as follows: Let $t_n=n^{1-\alpha}t$. Lemma \ref{cle1} shows that $\Theta(t_n)$ has the same law as $\frac{M_{R^{-1}(t_n)}}{Z(R^{-1}(t_n))}$. 
Moreover it is proved in Lemma 4.2 of \cite{BBS08} that $\frac{R^{-1}(t_n)}{t_n}\stackrel{P}{\to}\frac{1}{(\alpha-1)\alpha\Gamma(\alpha)}$, as $n$ goes to $\infty$. Hence one can compare the block sizes at time $t_n$ to those at time $R^{-1}((\alpha-1)\alpha\Gamma(\alpha)t_n).$  To this, we use the paintbox construction and the closeness between the measures $\frac{M_{t_n}}{Z(t_n)}$ and $\frac{M_{R^{-1}((\alpha-1)\alpha\Gamma(\alpha)t_n)}}{Z(R^{-1}((\alpha-1)\alpha\Gamma(\alpha)t_n))}$.  This idea can be executed through two steps.

\textbf{1) Analysis of the largest block size at time $t_n$ with the measure $\frac{M_{t_n}}{Z(t_n)}$:}
If $D(t_n)\neq 0$, let $\bar{J}_i(t_n)=\frac{J_i(t_n)}{Z(t_n)} $ for $1\leq i\leq D(t_n)$. Let $\{d_1(t_n), \cdots, d_{D(t_n)}(t_n)\}$ be an interval partition of $[0,1]$ such that the Lebesgue measure of $d_i(t_n)$ is $\bar{J}_i(t_n)$.  Build a partition of ${\mathbb{N}_n}$ from a paintbox associated with $\{d_1(t_n), \cdots, d_{D(t_n)}(t_n)\}$.  Let $N_i$ be the number of integers in $d_i(t_n) $ and $N=\max\{N_i: 1\leq i\leq D(t_n)\}$.

{\begin{lem}\label{lem:baseth1.5}
Let $x> 0$. Then

$1)$ $$\lim_{n\to\infty}\mathbb{P}(N\leq xn^{1/\alpha})=\exp(- \frac{(\alpha-1)tx^{-{\alpha}}}{\Gamma(2-\alpha)}).$$

$2)$ Let $0<y<x$. Then 
\begin{equation}\label{244}
\lim_{n\to\infty}\mathbb{P}(\exists i: J_{i}(t_n)< n^{\frac{1-\alpha}{\alpha}}y, N_i\geq xn^{\frac{1}{\alpha}})=0.
\end{equation}
\end{lem}}

\begin{proof}
$1)$ It is well known that if we throw a Poisson number of parameter $nZ(t_n)$ on $[0,1]$, the number of integers in $d_i(t_n)$, denoted by $\mathcal{N}_i$,  is a Poisson variable of parameter $nJ_i(t_n)$. Conditional on all $J_i(t_n)$'s, all $\mathcal{N}_i$'s are independent. Let $\mathcal{N}$ be the maximum of all $\mathcal{N}_i$'s.  Then, using {Lemmas \ref{lemtec}} and \ref{cle2},
\begin{align*}
&\mathbb{P}(\mathcal{N}\leq xn^{1/\alpha})=\mathbb{E}[\Pi_{i=1}^{D({t_n})}\mathbb{P}(\mathcal{N}_i\leq xn^{1/\alpha})]\longrightarrow \exp(-\gamma_t^{1-\alpha}\frac{x^{-\alpha}}{\Gamma(2-\alpha)})=\exp(- \frac{(\alpha-1)tx^{-{\alpha}}}{\Gamma(2-\alpha)}),\,\, n\to\infty.
\end{align*}
Lemma \ref{cle2} implies that $Z(t_n)$ tends in probability to $1$ as $n$ goes to infinity. Hence $N$ and $\mathcal{N}$ are close in the limit and standard comparison techniques allows to conclude. 

$2)$ As $Z(t_n)$ converges to $1$,  (\ref{244}) is equivalent to 
$$\lim_{n\to\infty}\mathbb{P}(\exists i: J_{i}(t_n)< n^{\frac{1-\alpha}{\alpha}}y, \mathcal{N}_i\geq xn^{\frac{1}{\alpha}})=0.$$
 Let $\tilde{\mathcal{N}}=\max\{\mathcal{N}_i: J_i(t_n)<n^{\frac{1-\alpha}{\alpha}}y\}$. It is necessary and sufficient to show that $\displaystyle \lim_{n\to\infty}\mathbb{P}(\tilde{\mathcal{N}}\geq xn^{\frac{1}{\alpha}})=0$. Notice that conditional on $J_i(t_n)$, $\mathcal{N}_i$ is a Posson variable with paprameter $nJ_i(t_n)$.  Let $\{P_1(yn^{\frac{1}{\alpha}}), P_2(yn^{\frac{1}{\alpha}}),\cdots\}$ be a sequence of i.i.d. Poisson variables with parameter $yn^{\frac{1}{\alpha}}$ and also independent of $D(t_{n})$. Then  
$$\mathbb{P}(\tilde{N}\geq xn^{\frac{1}{\alpha}})\leq \mathbb{P}\left(\max\{P_i(yn^{\frac{1}{\alpha}}): 1\leq i\leq D(t_n)\}\geq xn^{\frac{1}{\alpha}}\right)=1-\mathbb{E}[(\mathbb{P}(P_1(yn^{\frac{1}{\alpha}})<xn^{\frac{1}{\alpha}}))^{D(t_n)}].$$
Using (\ref{m+1}) and (\ref{alpha1}), one gets 
$$\mathbb{P}(P_1(yn^{\frac{1}{\alpha}})<xn^{\frac{1}{\alpha}})=1-o(\frac{1}{n}).$$
Meanwhile, Lemma \ref{cle2} tells that $\frac{D(t_n)}{n}$ converges in robability to $\gamma_t$ as $n$ goes to infinity. Hence we can conclude.
\end{proof}
\begin{rem}\label{save}
The key point to prove (\ref{244}) is that $Z(t_n)$ converges to $1$ in probability, $\frac{D(t_n)}{n}$ is asymptotically bounded by a positive value from above. The distribution of $\{J_i(t_n)\}_{1\leq i\leq D(t_n)}$ is not necessary to know. One can still find (\ref{244}) true if we replace $t_n$ by a random time and conditions for $Z(t_n)$ and $D(t_n)$ are satisfied at the same time. 
\end{rem}
\textbf{2) A tool lemma for the transfer from $\frac{M_{t_n}}{Z(t_n)}$ to $\frac{M_{R^{-1}((\alpha-1)\alpha\Gamma(\alpha)t_n)}}{Z(R^{-1}((\alpha-1)\alpha\Gamma(\alpha)t_n))}$ : }
Let $(A_1, \cdots, A_k)$ and $(B_1, \cdots, B_k)$ be two partitions of $[0,1]$ with $k\geq 1$.  
We throw away $n$ particles uniformly and independently on $[0,1]$ and regroup those within the same intervals of  $(B_1, \cdots, B_k)$, which gives a sequence of $k$ numbers $(N_{B_1}, \cdots, N_{B_k})$ such that $N_{B_i}$ is the number of particles located in $B_i$. 
We can obtain the law of this sequence in another way using $(A_1, \cdots, A_k)$:
We throw $n$ particles uniformly and independently on $[0,1]$.
Let $I:=\{i: 1\leq i\leq n, l(A_i)\leq l(B_i)\}$, where $l(\cdot)$ denotes the Lebesgue measure. If a particle falls in $A_i$ where $i\in I$, then move this particle to $B_i$. If a particle falls in $A_i$ where $i\in I^c$, then do the following:  we attach to this particle an independent Bernoulli variable with parameter $\frac{l(B_i)}{l({A_i})}$. If the Bernoulli variable gives $1$, then the particle is put into  $B_i$. Otherwise, this particle will be put into $B_j$ for $j\in I$ with probability 
\begin{equation}\label{exception}\frac{l(B_j)-l(A_j)}{\sum_{j\in I}(l(B_j)-l(A_j))}.\end{equation}
We denote by $N_{A_i}^B$ the new amount of particles in $B_i$.
We have the the following result.
\begin{lem}\label{abi}
The following identity in law holds.
$$(N^{B}_{A_1}, \cdots, N^{B}_{A_k})\stackrel{(d)}{=}(N_{B_1}, \cdots, N_{B_k}).$$
\end{lem}
\begin{proof}
Notice that only the measure of each element of $(A_1, \cdots, A_k)$ and $(B_1, \cdots, B_k)$ matters,  one can always assume that $[0,1]$ is divided in a way that $A_i$ is contained in $B_i$ for $i\in I$ and $B_i$ is contained in $A_i$ for $i\in I^c$. 
So if a particle is located in $A_i$ for $i\in I$, it is also located in $B_i$. But if a particle is located in $A_i$ for $i\in I^c$, with probability $\frac{l(B_i)}{l(A_i)}$ it is located in $B_i$. Assume that this particle is not located in $B_i$, then it must be in $\cup_{j\in I}B_j/A_j$. Using the uniformity of the throw, this particle falls in $B_j$ with probability (\ref{exception}).
\end{proof}

The above two steps allow to start the proof of Theorem \ref{max}. But before that, let us just recall some technical results from \cite{BBS08}. 
Let $\varepsilon>0$, $t>0$ and $t_n=n^{1-\alpha}t$. 
Let $t_{-}=(1-\varepsilon)t_n$, $t_{+}=(1+\varepsilon)t_n$ and $t_*=(\alpha-1)\alpha\Gamma(\alpha)t_n$.
Define the event $B_{1,t}:=\{t_{-}\leq R^{-1}(t_*)\leq t_{+}\}$.
It can be found in Lemma 4.2 of \cite{BBS08} that there exists a constant $C_{\ref{eq:lem4.2}}$ such that
\begin{equation}\label{eq:lem4.2}
 \P(B_{1,t})\geq1-C_{\ref{eq:lem4.2}}t_*\varepsilon^{-\alpha}.
\end{equation}
Also from Lemma 5.1 of \cite{BBS08}, there exists a constant $C_{\ref{eq:lem5.1a}}$ such that for all $a>0$, $t>0$ and $\eta>0$,
\begin{equation}\label{eq:lem5.1a}
q(a,t,\eta)=\P(\sup_{0\leq s\leq t}|Z(s,a)-a|>\eta)\leq C_{\ref{eq:lem5.1a}}(a+\eta)t\eta^{-\alpha}.
\end{equation}
Thus, if we define $B_{2,t}:=\{1-n^{\frac{1-\alpha}{2\alpha}}\leq Z(s)\leq 1+n^{\frac{1-\alpha}{2\alpha}}, \,\,\forall s\in[t_{-},t_{+}]\}$, we can obtain that
\begin{equation}\label{eq:lem5.1b}
 \P(B_{2,t})\geq1-C_{\ref{eq:lem5.1a}}t(1+\varepsilon)(1+n^{\frac{1-\alpha}{2\alpha}})n^{\frac{1-\alpha}{2}}.
\end{equation}

\begin{proof}[Proof of Theorem \ref{max}.]
{Lemma \ref{cle1} tells us that for any $s\geq 0$, we have 
 $$\frac{M_{R^{-1}(s)}}{Z(R^{-1}(s))}\stackrel{(d)}{=}\Theta(s).$$}Let $\pi$ be the partition of ${\mathbb{N}_n}$ obtained from a paintbox associated with $\frac{M_{R^{-1}(s)}}{Z(R^{-1}(s))}$. Then $\pi\stackrel{(d)}{=}\Pi^{(n)}(s)$. If $R^{-1}(s)\geq t_-$, we can as well at first build a partition  from a paintbox associated with $\frac{M_{t_{-}}}{Z(t_{-})}$ and then use Lemma \ref{abi} to get $\pi$.  This kind of construction is the key of this proof.

  
  
  

{ For $s\geq t_{-}$, one builds a partition of ${\mathbb{N}_n}$ from a paintbox associated with $(\frac{m_i(s)}{Z(s)},{1\leq i\leq D(t_{-})})$. 
We denote this partition by $V^{(n)}(s)=(V_1(s),V_2(s),\cdots, V_{D(t_{-})}(s))$ 
.
Let $I_i^{(n)}(s)$ be the number of particles in $V_i^{(n)}(s)$.}
 
{For $s\geq t_{-}$, let $M^{(n)}(s)=\sup\{I_i^{(n)}(s), 1\leq i\leq D(t_{-})\}$ be the size of the largest block of $V^{(n)}(s)$.  
Let $x>0$ and $\displaystyle B_{3,t}=\{ \exists k : I_k^n(t_{-})\geq xn^{\frac{1}{\alpha}}, J_k(t_-)\geq n^{\frac{2(1-\alpha)}{\alpha}},\sup_{t_{-}\leq s\leq t_{+}}|{m_k(s)-J_k(t_-)}|\leq \varepsilon J_k(t_{-})\}$.}

 On the event $ B_{3,t}$, we have that $M^{(n)}(t_-)\geq xn^{\frac{1}{\alpha}}$.
 Conditional on $B_{1,t}$ we can build the partition $V^{(n)}(R^{-1}(t_{*}))$ from a paintbox associated to the partition 
 $Z(t_-)^{-1}(J_1(t_-),\dots, J_{D(t_-)}(t_-))$ and Lemma \ref{abi}.
  Let $B(m,p)$ be a  binomial variable with parameters $m\geq 2 $ and $0\leq p\leq 1$. Lemma \ref{abi} implies
 \begin{align*} 
 &\mathbb{P}\left(M^{(n)}(R^{-1}(t_*))\geq (1-2\varepsilon)xn^{\frac{1}{\alpha}}|B_{1,t}\cap B_{2,t}\cap B_{3,t}\right) \nonumber\\
\geq&\mathbb{P}\left(B\left(\lceil xn^{\frac{1}{\alpha}}\rceil, \frac{m_k(R^{-1}(t_{*}))Z(t_{-})}{J_k(t_{-})Z(R^{-1}(t_{*}))}\wedge 1 \right)\geq (1-2\varepsilon)xn^{\frac{1}{\alpha}}|B_{1,t}\cap B_{2,t}\cap B_{3,t}\right)\\
\geq &\mathbb{P}\left(B\left(\lceil xn^{\frac{1}{\alpha}}\rceil, (1-\varepsilon)\frac{1-n^{\frac{1-\alpha}{2\alpha}}}{1+n^{\frac{1-\alpha}{2\alpha}}}\right)\geq (1-2\varepsilon)xn^{\frac{1}{\alpha}}\right)\\
=&\mathbb{P}\left((xn^{\frac{1}{\alpha}})^{-1}B\left(\lceil xn^{\frac{1}{\alpha}}\rceil, (1-\varepsilon)\frac{1-n^{\frac{1-\alpha}{2\alpha}}}{1+n^{\frac{1-\alpha}{2\alpha}}}\right)\geq (1-\varepsilon)-\varepsilon\right).
 \end{align*} 
 A law of large numbers argument implies that
\begin{equation}\label{right1}
 \mathbb{P}\left(M^{(n)}(R^{-1}(t_*))\geq (1-2\varepsilon)xn^{\frac{1}{\alpha}}|B_{1,t}\cap B_{2,t}\cap B_{3,t}\right)\geq1-\varepsilon
\end{equation}
for $n$ large enough.
Now observe that 
  \begin{align*}
  \mathbb{P}(B_{3,t})&= \mathbb{P}( \exists k : I_k^n(t_{-})\geq xn^{\frac{1}{\alpha}}, J_k(t_-)\geq n^{\frac{2(1-\alpha)}{\alpha}})\\
  &\quad \times  \mathbb{P}(\sup_{t_{-}\leq s\leq t_{+}}|m_k(s)-J_k(t_{-})|\leq\varepsilon J_k(t_{-})|\{ \exists k : I_k^n(t_{-})\geq xn^{\frac{1}{\alpha}}, J_k(t_-)\geq n^{\frac{2(1-\alpha)}{\alpha}})\\
  &= \mathbb{P}( \exists k : I_k^n(t_{-})\geq xn^{\frac{1}{\alpha}}, J_k(t_-)\geq n^{\frac{2(1-\alpha)}{\alpha}})\\
  &\quad \times (1-\mathbb{E}[q(J_k(t_{-}), t_{+}-t_{-}, \varepsilon J_k(t_{-}))| \exists k : I_k^n(t_{-})\geq xn^{\frac{1}{\alpha}}, J_k(t_-)\geq n^{\frac{2(1-\alpha)}{\alpha}})\\
  &\geq \mathbb{P}( \exists k : I_k^n(t_{-})\geq xn^{\frac{1}{\alpha}}, J_k(t_-)\geq n^{\frac{2(1-\alpha)}{\alpha}}){(1-2tC_{\ref{eq:lem5.1a}}n^{\frac{(1-\alpha)(2-\alpha)}{\alpha}}(1+\varepsilon)\varepsilon^{1-\alpha}}).
  \end{align*}
 Using Lemma \ref{lem:baseth1.5}, 
\begin{align*}
\mathbb{P}( \exists k : I_k^n(t_{-})\geq xn^{\frac{1}{\alpha}}, J_k(t_-)\geq n^{\frac{2(1-\alpha)}{\alpha}})&\sim \mathbb{P}( \exists k : I_k^n(t_{-})\geq xn^{\frac{1}{\alpha}})=\mathbb{P}(M^{(n)}(t_-)\geq n^{\frac{1}{\alpha}}x)\\
&\sim 1-\exp(- (1-\varepsilon)\frac{(\alpha-1)tx^{-{\alpha}}}{\Gamma(2-\alpha)}).
\end{align*}
In consequence,  \begin{align*}
  \underset{n\to\infty}{\liminf}\;\mathbb{P}(B_{3,t})
 \geq 1-\exp(- (1-\varepsilon)\frac{(\alpha-1)tx^{-{\alpha}}}{\Gamma(2-\alpha)})
\end{align*}
when $n$ tends to $\infty$.
  Then, thanks to \eqref{eq:lem4.2} and  \eqref{eq:lem5.1b},
  we deduce that 
  $$\underset{n\to\infty}{\liminf}\;\mathbb{P}(B_{1,t}\cap B_{2,t}\cap B_{3,t})\geq 1-\exp(- (1-\varepsilon)\frac{(\alpha-1)tx^{-{\alpha}}}{\Gamma(2-\alpha)}).$$
 Combining the latter with (\ref{right1}), we obtain 
  
\begin{align}\label{lowerbound}\underset{n\to\infty}{\liminf}\;\mathbb{P}\left(M^{(n)}(R^{-1}(t_*))\geq (1-2\varepsilon)xn^{\frac{1}{\alpha}}\right)\geq 1-\exp(- (1-\varepsilon)\frac{(\alpha-1)tx^{-{\alpha}}}{\Gamma(2-\alpha)}).\end{align}
  
  Next, we seek to find an upper bound for $\mathbb{P}\left(M^{(n)}(R^{-1}(t_*))\geq xn^{\frac{1}{\alpha}}\right)$. Conditional on $B_{1,t}$, we construct $V^{(n)}(t_{+})$ from $V^{(n)}(R^{-1}(t_{*}))$ using the method in Lemma \ref{abi}. Let 
$$B_{4,t}=B_{1,t}\cap\{\exists k: I_k^{(n)}(R^{-1}(t_*))\geq xn^{\frac{1}{\alpha}}, m_k(R^{-1}(t_{*}))\geq n^\frac{2(1-\alpha)}{\alpha}, \sup_{ R^{-1}(t_{*})\leq s\leq t_{+}}\frac{|m_k(s)-m_k(R^{-1}(t_{*}))|}{m_k(R^{-1}(t_{*}))}\leq \varepsilon \}.$$
 Similarly as for the lower bound,  
   \begin{align}\label{right2}
 &\mathbb{P}\left(M^{(n)}(t_{+})\geq (1-2\varepsilon)xn^{\frac{1}{\alpha}}|B_{2,t}\cap B_{4,t}\right) \nonumber\\
\geq&\mathbb{P}\left(B\left(\lceil xn^{\frac{1}{\alpha}}\rceil, \frac{Z(R^{-1}(t_{*}))m_k(t_+)}{Z(t_{+})m_k(R^{-1}(t_{*}))}\wedge 1 \right)\geq (1-2\varepsilon)xn^{\frac{1}{\alpha}}|B_{2,t}\cap B_{4,t}\right)\nonumber\\
\geq &\mathbb{P}\left(B\left(\lceil xn^{\frac{1}{\alpha}}\rceil, (1-\varepsilon)\frac{1-n^{(1-\alpha)/\alpha}}{1+n^{(1-\alpha)/\alpha}}\right)\geq (1-2\varepsilon)xn^{\frac{1}{\alpha}}\right)\longrightarrow 1.
 \end{align}
  
  Using the strong Markov property of the CSBP and (\ref{eq:lem5.1a}), we have 
  
  \begin{align}\label{b}
  \mathbb{P}(B_{4,t})&=\mathbb{P}(B_{1,t}\cap\{\exists k: I_k^{(n)}(R^{-1}(t_*))\geq xn^{\frac{1}{\alpha}}, m_k(R^{-1}(t_{*}))\geq n^\frac{2(1-\alpha)}{\alpha}\})\\
&\quad \times {(1-2tC_{\ref{eq:lem5.1a}}n^{\frac{(1-\alpha)(2-\alpha)}{\alpha}}(1+\varepsilon)\varepsilon^{1-\alpha})}
  \end{align}
  Notice that using (\ref{eq:lem5.1a}), in the sense of convergence of probability
 $$\lim_{n\to\infty}\sup_{t_-\leq s\leq t_+}Z(s)=\lim_{n\to\infty}\inf_{t_-\leq s\leq t_+}Z(s)=1$$
Together with (\ref{eq:lem4.2}), we get the following convergence in probability
$$\lim_{n\to\infty}Z(R^{-1}(t_*))=1.$$
Remark \ref{save1} tells that $D(t)$ is decreasing on t. Under $B_{1,t}$, $D(t_-)\leq D(R^{-1}(t_*))\leq D(t_+)$. It is then easy to deduce that $\frac{D(R^{-1}(t_*))}{n}$ is asymptotically bounded from above by a certain positive number.  
Now we can apply Remark \ref{save} and get 
\begin{equation}\label{t-+}\mathbb{P}(B_{4,t})=\mathbb{P}(\exists k: I_k^{(n)}(R^{-1}(t_*))\geq xn^{\frac{1}{\alpha}})+o(1)=\mathbb{P}(M^{(n)}(R^{-1}(t_*))\geq xn^{\frac{1}{\alpha}})+o(1).\end{equation}
Using (\ref{right2}), (\ref{eq:lem5.1b}) and (\ref{t-+})
    \begin{align}\label{upperbound}
    &\limsup_{n\longrightarrow \infty} \mathbb{P}(M^{(n)}(R^{-1}(t_{*}))\geq xn^{\frac{1}{\alpha}})\nonumber\\
    \leq& \lim_{n\longrightarrow \infty}\mathbb{P}(M^{(n)}(t_{+})
    \geq (1-2\varepsilon)xn^{\frac{1}{\alpha}})\nonumber\\
    =&1-\exp({-(x(1-2\varepsilon))^{-\alpha}\frac{(\alpha-1)t(1+\varepsilon)}{\Gamma(2-\alpha)}}).
    \end{align}
Since $\varepsilon$ can be arbitrarily small,  (\ref{lowerbound}) and (\ref{upperbound}) allow to conclude.\end{proof}

Finally, observe that Corollary \ref{max1}
is obtained from a combination of Lemma \ref{proba} and Theorem \ref{max}.


\begin{thebibliography}{10}

\bibitem{Ar04}
E.~\'Arnason.
\newblock Mitochondrial cytochrome b dna variation in the high-fecundity
  atlantic cod: trans-atlantic clines and shallow gene genealogy.
\newblock {\em Genetics}, 166(4):1871--1885, 2004.

\bibitem{BBS07}
J.~Berestycki, N.~Berestycki and J.~Schweinsberg.
\newblock Beta-coalescents and continuous stable random trees.
\newblock {\em Ann. Probab.}, 35(5):1835--1887, 2007.

\bibitem{BBS08}
J.~Berestycki, N.~Berestycki and J.~Schweinsberg.
\newblock Small-time behavior of {B}eta-coalescents.
\newblock {\em Ann. Inst. H. Poincar\'e Probab. Stat.}, 44(2):214--238, 2008.

\bibitem{BLG00}
J.~Bertoin and J-F.~Le~Gall.
\newblock The {B}olthausen-{S}znitman coalescent and the genealogy of
  continuous-state branching processes.
\newblock {\em Probab. Theory Related Fields}, 117(2):249--266, 2000.

\bibitem{B89}
N.H.~Bingham, C.M.~Goldie and J.L.~Teugels.
 Regular variation-
 \newblock {\em
Encyclopedia of Mathematics and its Applications}, 27. Cambridge University Press, Cambridge. MR898871.
 1987.

\bibitem{les705}
M.~Birkner, J.~Blath, M.~Capaldo, A.~M. Etheridge, M.~M{\"o}hle,
  J.~Schweinsberg and A.~Wakolbinger.
\newblock Alpha-stable branching and {B}eta-coalescents.
\newblock {\em Electron. J. Probab.}, 10(9):303--325, 2005.

\bibitem{BF05}
M.G.B.~Blum and O.~Fran{\c{c}}ois.
\newblock Minimal clade size and external branch length under the neutral
  coalescent.
\newblock {\em Adv. in Appl. Probab.}, 37(3):647--662, 2005.

\bibitem{BS98}
E.~Bolthausen and A.~Sznitman.
\newblock On Ruelle's probability cascades and an abstract cavity method.
\newblock {\em Commun. Math. Phys}, 197(2):247--276, 1998.

\bibitem{Bo94}
J.~Boom, E.~Boulding and A.~Beckenbach.
\newblock Mitochondrial dna variation in introduced populations of pacific
  oyster, crassostrea gigas, in british columbia.
\newblock {\em Canadian Journal of Fisheries and Aquatic Sciences},
  51(7):1608--1614, 1994.

\bibitem{CNK07}
A.~Caliebe, R.~Neininger, M.~Krawczak and U.~R{\"o}sler.
\newblock On the length distribution of external branches in coalescence trees:
  genetic diversity within species.
\newblock {\em Theoret. Population Biol.}, 72(2):245--252, 2007.

\bibitem{KIA12}
I.~Dahmer, G.~Kersting and A.~Wakolbinger.
\newblock The total external branch length of Beta-coalescents.
\newblock {\em arXiv preprint arXiv:1212.6070}.
\newblock To appear in {\em Combinatorics, Probability and Computing}.

\bibitem{DDSJ08}
J-F.~Delmas, J-S.~Dhersin and A.~Siri-J\'egousse.
\newblock Asymptotic results on the length of coalescent trees.
\newblock {\em Ann. Appl. Probab.}, 18(3):997--1025, 2008.

\bibitem{DFSY}
J-S.~Dhersin, F.~Freund, A.~Siri-J{\'e}gousse and L.~Yuan.
\newblock On the length of an external branch in the beta-coalescent.
\newblock {\em Stoch. Proc. Appl.}, 123:1691--1715, 2013.

\bibitem{DM13}
J-S.~Dhersin and M.~M{\"o}hle.
\newblock On the external branches of coalescents with multiple collisions.
\newblock {\em Electron. J. Probab}, 18(40):1--11, 2013.

\bibitem{DK99}
P.~Donnelly and T.G.~Kurtz.
\newblock Particle representations for measure-valued population models.
\newblock {\em Ann. Probab.}, 27(1):166--205, 1999.

\bibitem{El06}
B.~Eldon and J.~Wakeley.
\newblock Coalescent processes when the distribution of offspring number among
  individuals is highly skewed.
\newblock {\em Genetics}, 172:2621--2633, 2006.

\bibitem{FM09}
F.~Freund and M.~M{\"o}hle.
\newblock On the time back to the most recent common ancestor and the external
  branch length of the {B}olthausen-{S}znitman coalescent.
\newblock {\em Markov Process. Related Fields}, 15(3):387--416, 2009.

\bibitem{FA13}
F.~Freund and A.~Siri-J{\'e}gousse.
\newblock Minimal clade size in the Bolthausen-Sznitman coalescent.
\newblock To appear in {\em J. Appl. Probab.}.

\bibitem{GIM08}
A.~Gnedin, A.~Iksanov and M.~M{\"o}hle.
\newblock On asymptotics of exchangeable coalescents with multiple collisions.
\newblock {\em J. Appl. Probab.}, 45:1186--1195, 2008.

\bibitem{He94}
D.~Hedgecock.
\newblock Does variance in reproductive success limit effective population
  sizes of marine organisms?
\newblock In {\em Genetics and Evolution of Aquatic Organisms},
  Chapman and Hall, London, 1222--1344, 1994.

\bibitem{HJ08}
H-K.~Hwang and S.~Janson.
\newblock Local limit theorems for finite and infinite urn models.
\newblock {\em Ann. Probab.}, 36(3):992--1022, 2008.

\bibitem{Ker12}
G.~Kersting.
\newblock The asymptotic distribution of the length of beta-coalescent trees.
\newblock {\em Ann. Appl. Probab.}, 22(5):2086--2107, 2012.

\bibitem{Kin82a}
J.F.C.~Kingman.
\newblock The coalescent.
\newblock {\em Stoch. Proc. Appl.}, 13(3):235--248, 1982.

\bibitem{Kin82b}
J.F.C.~Kingman.
\newblock Exchangeability and the evolution of large populations.
\newblock In {\em Exchangeability in probability and statistics ({R}ome,
  1981)}, North-Holland, Amsterdam, 97--112, 1982.

\bibitem{Kin82c}
J.F.C.~Kingman.
\newblock On the genealogy of large populations.
\newblock {\em J. Appl. Probab.}, 19(1):27--43, 1982.

\bibitem{M10}
M.~M{\"o}hle.
\newblock Asymptotic results for coalescent processes without proper
  frequencies and applications to the two-parameter Poisson-Dirichlet
  coalescent.
\newblock {\em Stoch. Proc. Appl.},
  120(11):2159--2173, 2010.

\bibitem{Pit99}
J.~Pitman.
\newblock Coalescents with multiple collisions.
\newblock {\em Ann. Probab.}, 27(4):1870--1902, 1999.

\bibitem{RBY04}
E.~Rauch and Y.~Bar-Yam.
\newblock Theory predicts the uneven distribution of genetic diversity within
  species.
\newblock {\em Nature}, 431:449--452, 2004.

\bibitem{Sag99}
S.~Sagitov.
\newblock The general coalescent with asynchronous mergers of ancestral lines.
\newblock {\em J. Appl. Probab.}, 36(4):1116--1125, 1999.

\bibitem{Sch00de}
J.~Schweinsberg.
\newblock A necessary and sufficient condition for the $\Lambda$-coalescent to
  come down from infinity.
\newblock {\em Electron. Comm. Probab.}, 5:1--11, 2000.

\bibitem{Sch03}
J.~Schweinsberg.
\newblock Coalescent processes obtained from supercritical {G}alton-{W}atson
  processes.
\newblock {\em Stoch. Proce. Appl.}, 106(1):107--139, 2003.

\bibitem{Sla68}
R.~Slack.
\newblock A branching process with mean one and possibly infinite variance.
\newblock {\em Probab. Theory Related Fields}, 9(2):139--145, 1968.

\bibitem{Yuan13}
L.~Yuan.
\newblock On the measure division construction of $\Lambda$-coalescents.
\newblock {\em arXiv preprint arXiv:1302.1083}, 2013.
\newblock To appear in {\em Markov Process. Related Fields}.

\end{thebibliography}
\end{document}